\documentclass[a4paper,10pt]{article}
\usepackage{mypackages}
\usepackage{mymacros}
\usepackage{mystyle}

\title{Coarse Cohomology with Twisted Coefficients}
\author{Elisa Hartmann\thanks{Department of Mathematics, Karlsruhe Institute of 
Technology}}

\begin{document}

\maketitle

\begin{abstract}
 To a coarse structure we associate a Grothendieck topology which is determined by coarse covers. A coarse map between coarse spaces gives rise to a morphism of Grothendieck topologies. This way we define sheaves and sheaf cohomology on coarse spaces. We obtain that sheaf cohomology is a functor on the coarse category: if two coarse maps are close they induce the same map in cohomology. There is a coarse version of a Mayer-Vietoris sequence and for every inclusion of coarse spaces there is a coarse version of relative cohomology. Cohomology with constant coefficients can be computed using the number of ends of a coarse space. 
\end{abstract}

\tableofcontents

\section{Introduction}

In this paper we introduce a new invariant on coarse spaces called \emph{coarse 
sheaf cohomology}. This cohomology theory is not based on the topology of open 
sets of a metric space but on coarse covers. Coarse covers are finite families of subsets of a metric space which satisfy a boundedness condition. The collection of coarse covers of a coarse space determine a Grothendieck topology, coarse maps serve as morphisms of Grothendieck topologies. This way we define sheaves on coarse spaces and coarse maps transfer sheaves between spaces.

 Of particular interest to us are locally constant sheaves since cohomology with constant coefficient an abelian group $A$ is a functor on coarse spaces. Remarkable is the connection between the number of ends of a coarse space and its constant coefficient.

Without more advanched methods (which has been done in a follow-up research) cohomology of non-trivial examples is hard to compute. If $A$ is an infinite group, $\Z$ for example, then cohomology with 
coefficient $\Z$ on $\Z_+$, the positive integers is highly nontrivial 
\cite{Keesling1994}. This problem does not occur if $A$ is finite. Finite coefficients produce interesting cohomology groups.

\subsection{Approach}

Our purpose is to pursue an algebraic geometry approach to coarse geometry. We present sheaf cohomology on coarse spaces and study coarse spaces by coarse cohomology with twisted coefficients. The method is based on the theory on Grothendieck topologies.

Note that sheaves on Grothendieck topologies and sheaf cohomology theory have been applied in a number of areas and have lead to many breakthroughs on previously unsolved problems. As stated in~\cite{Mclarty2007} one can understand a mathematical problem by
\begin{enumerate}
 \item finding a mathematical world natural for the problem.
 \item Expressing your problem cohomologically.
 \item The cohomology of that world may solve your problem. 
\end{enumerate}
That way we can apply general theory on sheaf cohomology for tackling 
previously unsolved problems and studying notions which are quite well known. 

\subsection{What is Coarse Geometry?}
The topic coarse geometry studies metric spaces from a large scale point of 
view. We want to examine the global structure of metric spaces. One way to 
approach this problem is by forgetting small scale structure. The coarse 
category consists of coarse spaces as objects and coarse maps modulo closeness 
as morphisms.

Now coarse maps preserve the coarse structure of a space in the coarse 
category. A coarse structure is made of \emph{entourages} which are 
surroundings of the diagonal. For us metric spaces are the main objects of 
study. If $X$ is a metric space a subset $E\s X^2$ is an entourage if 
\[
\sup_{(x,y)\in E}d(x,y)<\infty.
\]
The exact opposite of a coarse space and coarse geometry of metric spaces are 
uniform spaces and the uniform topology of a metric space. Like coarse spaces 
uniform spaces are defined via surroundings of the diagonal. Uniform entourages 
get smaller though while coarse entourages get larger the sharper the point of 
view.

Many algebraic properties of infinite finitely generated groups are hidden in 
the geometry of their Cayley graph. To a finitely generated group is associated 
the word length with regard to a generating set. Note that the metric of the 
group depends on the choice of generating set while the coarse structure 
associated to the word length metric is independent of the choice of generating 
set. Note that group homomorphisms are instances of coarsely uniform maps 
between groups and a group isomorphism is an instance of a coarse 
equivalence between groups. It is very fruitful to group theory to consider 
infinite finitely generated groups as coarse objects; these will be a source of 
examples for us.

Note the examples $\R^n$ and $\Z^n$ both are coarse spaces induced by a metric, for $\R^n$ it is the euclidean metric and for $\Z^n$ the metric is induced by the group $(\Z^n,+)$. Now $\Z^n$ and $\R^n$ look entirely different on small scale they are the same on large scale though. There is a coarse equivalence $\Z^n\to \R^n$.

\subsection{Background and related Theories}

Nowadays it is hard to embrace all cohomology theory and other theories in the coarse category because of the diversity of the toolsets used. 

A \emph{cohomology theory} assigns an abelian group with a space, in a functorial manner. There are classical examples like \v Cech cohomology, simplicial homology, $\ldots$ etc. which all fit in a general framework. The standard choice in the topological category are the Eilenberg-Steenrod axioms. They consist of 5 conditions which characterize singular cohomology on topological spaces. A \emph{generalized cohomology theory} is a sequence of contravariant functors $(H^n)_n$ from the category of pairs of topological spaces $(X,A)$ to the category of abelian groups equipped with natural transformations
\[
\delta:H^n(A,\emptyset)\to H^{n+1}(X,A)
\]
for $n\in \N$, such that
\begin{enumerate}
\item \emph{Homotopy}: If $f_1,f_2:(X,A)\to (Y,B)$ are homotopic morphisms then they induce isomorphic maps in cohomology.
\item \emph{Excision}: If $(X,A)$ is a pair and $U\s A$ a subset such that $\bar U\s A^\circ$ then the inclusion 
\[
i:(X\ohne U,A\ohne U)\to (X,A)
\]
induces an isomorphism in cohomology.
\item \emph{Dimension}: The cohomology of the point is concentrated in degree $0$.
\item \emph{Additivity}: If $X=\bigsqcup_\alpha X_\alpha$ is a disjoint union of topological spaces then
\[
 H^n(X,\emptyset)=\prod_\alpha H^n(X_\alpha,\emptyset).
\]
\item \emph{Exactness}: Every pair of topological spaces $(X,A)$ induces a long exact sequence in cohomology:
\f{
\cdots\to& H^n(X,A)\to H^n(X,\emptyset)\to H^n(A,\emptyset)\\
\to & H^{n+1}(X,A)\to \cdots.
}
\end{enumerate}

We are interested in theories that are functors on coarse spaces and coarse maps. Let us first recall the standard theories.

There are a number of cohomology theories in the coarse category we present two of them which are the most commonly used ones. We first present the most basic facts about \emph{controlled operator $K$-theory} and Roe's \emph{coarse cohomology}.

We begin with a covariant invariant $K_*(C^*(\cdot))$ on proper metric spaces called \emph{controlled $K$-theory}. Note that if a proper metric space $B$ is bounded then it is compact. Then~\cite[Lemma~6.4.1]{Higson2000} shows
 \[
  K_p(C^*(B))=\begin{cases}
  \Z & p=0\\
  0 & p=1.
  \end{cases}
 \]
 There is a notion of flasque spaces for which controlled $K$-theory vanishes. An exemplary example is $\Z_+$; in~\cite[Lemma~6.4.2]{Higson2000} it is shown that
 \[
  K_*(C^*(\Z_+))=0.
 \]
 The above is used in order to compute the controlled $K$-theory of $\Z^n$:
 \[
  K_p(C^*(\Z^n))=\begin{cases}
  \Z & p\equiv n \mod 2\\
  0 & p\equiv n+1 \mod 2
  \end{cases}
 \]
  which is~\cite[Theorem~6.4.10]{Higson2000}. 
 The notion of Mayer-Vietoris sequence is adapted to this setting: If there are two subspaces $A,B$ of a coarse space and if they satisfy the coarse excisive property which is introduced in~\cite{Higson1993} then~\cite[Lemmas~1,2; Section~5]{Higson1993} combine to a Mayer-Vietoris sequence in controlled $K$-theory.
  There is a notion of homotopy for the coarse category which is established in~\cite{Higson1994}. Then~\cite[Theorem~5.1]{Higson1994} proves that controlled K-Theory is a coarse homotopy invariant.

Let us now consider \emph{coarse cohomology} $HX^*(\cdot;A)$ which for $A$ an abelian group is a contravariant invariant on coarse spaces.
The ~\cite[Example~5.13]{Roe2003} notes that if a coarse space $B$ is bounded then 
 \[
  HX^q(B;A)=\begin{cases}
  A & q=0\\
  0 & \mbox{otherwise}.
  \end{cases}
 \]
 Now the space $\Z^n$ reappears as an example in~\cite[Example~5.20]{Roe2003}:
 \[
  HX^q(\R^n;\R)=\begin{cases}
  0 & q\not=n\\
  \R & q=n
  \end{cases}
 \]
 Whereas another example is interesting: the~\cite[Example~5.21]{Roe2003} shows that if $G$ is a finitely generated group then there is an isomorphism
\[
 HX^*(G;\Z)=H^*(G;\Z[G]).
\]
Here the right side denotes group cohomology.
 In order to compute coarse cohomology there is one method: We denote by $H^*_c(X;A)$ the cohomology with compact supports of $X$ as a topological space. There is a character map
\[
 c:HX^q(X;A)\to H^q_c(X;A)
\]
By~\cite[Lemma~5.17]{Roe2003} the character map $c$ is injective if $X$ is a proper coarse space which is topologically path-connected. Now~\cite[Theorem~5.28]{Roe2003} states: If $R$ is a commutative ring and $X$ is a uniformly contractible proper coarse space the character map for $R$-coefficients is an isomorphism. 

 In the course of this article we will design a new cohomology theory on coarse spaces. It has all the pros of the existing coarse cohomology theories and can be compared with them. The main purpose of this work is to design computational tools for the new theory and compute cohomology of a few exemplary examples. 

Our main tool will be \emph{sheaf cohomology theory}, which we now recall. If $X$ is a coarse space then $\sheaves X$ denotes the abelian category of sheaves of abelian groups on $X$. Note that $\sheaves X$ has enough injectives. Then the global sections functor
\[
\sheaff\mapsto \Gamma(X,\sheaff)
\]
is a left exact functor between abelian categories $\sheaves X$ and $\abel$, the category of abelian groups. The right derived functors are the sheaf cohomology functors. If $\sheaff$ is a sheaf on $X$ then $\cohomology * X \sheaff$ denotes coarse cohomology with twisted coefficients with values in $\sheaff$.

There are many ways to compute sheaf cohomology. One of them uses acyclic resolutions. Now every sheaf $\sheaff$ on a coarse space $X$ has an injective resolution and injective sheaves are acyclic. Thus there exists a resolution
\[
0\to\sheaff\to \sheafi_0\to\sheafi_1\to \sheafi_2\to \cdots
\]
with acyclics $\sheafi_q$, $q\ge 0$. Then the sheaf cohomology groups $\cohomology q X \sheaff$ are the cohomology groups of the following complex of abelian groups
\[
0\to \sheafi_0(X)\to \sheafi_1(X)\to \sheafi_2(X)\to\cdots.
\]

We can also compute sheaf cohomology by means of \v Cech cohomology. If $(U_i)_{i\in I}$ is a \emph{coarse cover} of a subset $U\s X$ and $\sheaff$ an abelian presheaf on $X$ then the group of $q$-cochains is
\[
C^q(\{U_i\to U\}_i,\sheaff)=\prod_{(i_0,\ldots,i_q)\in I^{q+1}}\sheaff(U_{i_0}\cap\cdots\cap U_{i_q})
\]
The coboundary operator $d^q:C^q(\{U_i\to U\}_i,\sheaff)\to C^{q+1}(\{U_i\to U\}_i,\sheaff)$ is defined by
\[
(d^qs)_{i_0,\ldots,i_{q+1}}=\sum_{\nu=0}^{q+1}(-1)^\nu s_{i_0,\ldots,\hat i_\nu,\ldots i_{q+1}}|_{i_0,\ldots,i_{q+1}}
\]
Then $C^*(\{U_i\to U\}_i,\sheaff)$ is a complex and $\cohomology *{\{U_i\to U\}_i}\sheaff$ is defined to be its cohomology. Now sheaf cohomology can be computed:
\[
\cohomology q U \sheaff=\varinjlim_{\{U_i\to U\}_i}\cohomology q {\{U_i\to U\}_i}\sheaff.
\]

In good circumstances we can compute sheaf cohomology using an acyclic cover. If $(U_i)_{i\in I}$ is a coarse cover of a coarse space $X$ and $\sheaff$ a sheaf on $X$ and if for every nonempty $\{i_1,\ldots,i_n\}\s I$, $q>0$ we have that
\[
\cohomology q {U_{i_1}\cap\cdots\cap U_{i_n}} \sheaff=0
\]
then already
\[
\cohomology q X \sheaff =\cohomology q {\{U_i\to U\}_i} \sheaff
\]
for every $q\ge 0$.

Note that homotopy also plays an important part when computing sheaf cohomology.

\subsection{Main Contributions}

The general idea of this work is to transfer toolsets from other topics like 
algebraic topology and algebraic geometry and use them in the coarse category. 
The cohomology theory we are aiming at has its roots in algebraic geometry. 
First let us note a few aspects which distinguishes the new theory.

 There has been much effort in establishing axioms for cohomology theories in 
the coarse category. In \cite{Bunke2017} has been proposed a choice of axioms 
for coarse cohomology theories. Now we will test our theory against the 
Eilenberg-Steenrod axiom system. The new theory satisfies similar properties 
which are going to be discussed in the following list

\begin{enumerate}
\item \emph{Homotopy:} The relation close on coarse maps can be regarded as a notion of homotopy on the coarse category. Sheaf cohomology on coarse spaces 
is an invariant modulo close.
\item \emph{Excision:} Subsection~\ref{subsec:loccoh} presents local cohomology in the coarse category.
\item \emph{Dimension:} The space $\mathbb Z_+$ can be understood as the 
coarse equivalent of a point. It is however not acyclic for general 
coefficients. If the spaces $\mathbb Z^n$ are understood as representatives  for 
dimension then coarse cohomology with twisted coefficients sees dimension.
\item \emph{Additivity:} Sheaf cohomology sees coproducts, see subsection~\ref{subsec:definition}.
\item \emph{Exactness:} Subsection~\ref{subsec:MV} presents a coarse version of the Mayer-Vietoris sequence.
\end{enumerate}

Now why are there so many powerful results is one of the most natural questions 
we can ask. The main reason is, that typically sheaf cohomology is a powerful 
tool in a number of areas. Examples are de Rham cohomology in differential 
geometry, singular cohomology for nice enough spaces in algebraic topology and 
\'etale cohomology in algebraic geometry.

 A Grothendieck topology is the least amount of data needed to define  sheaves 
and sheaf cohomology. And that is where we start. We define the Grothendieck 
topology of coarse covers associated to a coarse space in 
Definition~\ref{defn:grothendiecktopology}. Then we discover in 
Lemma~\ref{lem:morphismoftopologies} that coarse maps give rise to a morphism of 
topologies. That is all the information that we need to use the powerful 
machinery of sheaf cohomology. 

Then we obtain the first important result: if two coarse maps are close then they induce isomorphic maps in cohomology with twisted coefficients. This is Theorem~\ref{thm:functor}.

\begin{thma}
Coarse cohomology with twisted coefficients is a functor on coarse spaces and coarse maps modulo closeness.
\end{thma}
Thus coarsely equivalent coarse spaces have the same cohomology.

The coarse equivalent of a trivial space is either the empty set or a bounded space or both. If $B$ is a bounded space then for every coefficient $\sheaff$ on $B$:
\[
 \cohomology * B \sheaff=0
\]
which is a result of Example~\ref{ex:bounded}.

Some computional tools we recognize from algebraic topology can be adopted for our setting. The Chapter~\ref{subsec:MV} presents a coarse version of Mayer-Vietoris:

\begin{thma}\name{Mayer-Vietoris}
Let $X$ be a coarse space and $A,B$ two subsets that coarsely cover $X$. Then there is an exact sequence in cohomology
 \f{
 \cdots&\to \cohomology {i-1} {A\cap B}\sheaff\to \cohomology {i} {A\cup B}\sheaff\to\cohomology {i} 
{A}\sheaff\times\cohomology i B \sheaff\\
 &\to \cohomology {i} {A\cap B}\sheaff\to\cdots
 }
 for every sheaf $\sheaff$ on $X$.
\end{thma}

The Chapter~\ref{subsec:loccoh} discusses relative cohomology in the coarse category.

\begin{thma}
 Let $Z\s X$ be a subspace of a coarse space and let $Y=X\ohne Z$. Then there is a long exact sequence
 \[
  0\to \cohomology 0 {U}{\Gamma_Z(\sheaff)}\to \cohomology 0 {U}\sheaff\to \cohomology 0 {U} 
{\sheaff|_Y}\to \cohomology 1 {U} {\Gamma_Z(\sheaff)}\to \cdots
 \]
for every subset $U\s X$ and every sheaf $\sheaff$ on $X$.
\end{thma}

In Chapter~\ref{sec:cc} constant coefficients on coarse spaces are introduced. 
If $A$ is an abelian group and $X$ is a coproduct of $n$ unbounded coarse 
spaces then its number of ends is at least $n$ and $A(X)\ge A^n$. In fact if $X$ 
does not have finitely many ends then $A(X)=A^{\oplus \N}$. This is discussed 
in Theorem~\ref{thm:constantcoeffs}.

First let us note that $\Z_+$ is imperfect as a coarse version of a point as it 
is not a final object and does not have trivial cohomology. While $\cohomology q 
{\Z_+} A=0$ for $q\ge 2$ and every constant coefficient $A$, the cohomology in 
degree 1,
\[
\cohomology 1 {\Z_+} \Z\not=0
\]
is nontrivial for $\Z$-coefficients. By Lemma~\ref{lem:zplusmanyended} every 
unbounded subset $U\s \Z_+$ has either infinitely many ends or the inclusion 
$U\to \Z_+$ is coarsely surjective.

Note if $X$ is any metric space we can find a sequence $(x_i)_i\s X$ with 
$d(x_i,x_j)>i$ for $j<i$. This space is discrete which means every entourage on 
$(x_i)_i$ has finitely many offdiagonal points. Its Higson compactification is 
homeomorphic to the Stone-\v Cech compactification of the natural numbers.

This makes it extremely hard to compute cohomology of specific examples. The 
cohomology of a bounded metric space is trivial of course since every sheaf on 
it vanishes. If $X$ is discrete or $\asdim(X)=0$ then the cohomology of $X$ is 
acyclic.

\subsection{Outline}
Now we indicate an outline of the chapters that are going to appear.
\begin{itemize}
 \item Chapters~\ref{CCwtCsec:intro},\ref{sec:coents} serve as an introduction.
 \item in Chapter~\ref{sec:lac} we construct new spaces out of old ones.
 \item Chapter~\ref{sec:ct} studies the coarse cohomology theory with twisted 
coefficients
\item Chapter~\ref{sec:cc} presents cohomology with constant coefficients.
\end{itemize}

\section{The Coarse Category}
\label{CCwtCsec:intro}
The following chapter introduces coarse spaces and coarse maps between coarse spaces. It has been kept as short as possible, giving only the most basic definitions needed for understanding this paper. All this information can be found in~\cite[chapter~2]{Roe2003}.

\subsection{Coarse Spaces}
\label{sec:coarsedefs}
\begin{defn}\name{inverse, product}
 Let $X$ be a set and let $E$ be a subset of $X^2$. Then the \emph{inverse}
$E^{-1}$ is 
defined by 
 \[
  E^{-1}=\{(y,x)|(x,y)\in E\}.
 \]
 
 A set $E$ is called \emph{symmetric} if $E=E^{-1}$.

 For two subsets $E_1,E_2\s X^2$ the \emph{product} $E_1\circ E_2$ is given by
 \[
  E_1\circ E_2=\{(x,z)|\exists y: (x,y)\in E_1, (y,z)\in E_2\}. 
 \]
\end{defn}

\begin{defn}\name{coarse structure}
\label{defn:coarsestructure}
 Let $X$ be a set. A \emph{coarse structure} on $X$ is a collection of subsets $E\s X^2$ which will be referred as \emph{entourages} which follow the following axioms:
 \begin{enumerate}
  \item  the diagonal $\Delta_X=\{(x,x)|x\in X\}$ is an entourage;
  \item  if $E$ is an entourage and $F\s E$ a subset then $F$ is an entourage;
  \item  if $F,E$ are entourages then $F\cup E$ is an entourage;
  \item  if $E$ is an entourage then the inverse $E^{-1}$ is an entourage;
  \item  if $E_1,E_2$ are entourages then their product $E_1\circ E_2$ is an entourage.
 \end{enumerate}
 The set $X$ together with the coarse structure on $X$ will be called a \emph{coarse space}.
\end{defn}

\begin{defn}\name{connected}
 A coarse space $X$ is \emph{connected} if 
 \begin{enumerate}
  \item [6.] for every points $x,y\in X$ the set $\{(x,y)\}\s X^2$ is an entourage. 
 \end{enumerate}
In the course of this paper all coarse spaces are assumed to be connected unless said otherwise. 
\end{defn}

\begin{defn}\name{bounded set}
 Let $X$ be a coarse space. A subset $B\s X$ is called \emph{bounded} if $B^2$ is an entourage.
 \end{defn}

\begin{defn}
Let $X$ be a set and let $K\s X$ and $E\s X^2$ be subsets. One writes
\[
 E[K]=\{x|\exists y\in K:(x,y)\in E\}.
\]
In case $K$ is just a set containing one point $p$, we write $E_p$ for $E[\{p\}]$ (called a section).
\end{defn}

\begin{lem}
\label{lem:boundedsets}
Let $X$ be a coarse space. 
\begin{itemize}
\item If $B_1,B_2\s X$ are bounded then $B_1\times B_2$ is an entourage and $B_1\cup B_2$ is bounded.
\item For every bounded subset $B\s X$ and entourage $E$ the set $E[B]$ is bounded. 
\end{itemize}
\end{lem}
\begin{proof}
\begin{itemize}
\item Fix two points $b_1\in B_1$ and $b_2\in B_2$ then $(b_1,b_2)$ is an entourage in $X$. Thus
\[
B_1^2\circ (b_1,b_2)\circ B_2^2=B_1\times B_2
\]
is an entourage. Now
\[
(B_1\cup B_2)^2=B_1^2\cup B_1\times B_2 \cup B_2\times B_1 \cup B_2^2
\]
is an entourage, thus $B_1\cup B_2$ is bounded.
\item Note that 
\[
E\circ B^2=E[B]\times B
\]
is an entourage.
\end{itemize}
\end{proof}

\begin{rem}
Note that an intersection of coarse structures is again a coarse structure.
\begin{itemize}
\item If $X$ is a set and $\delta$ a collection of subsets of $X^2$ then the smallest coarse structure $\varepsilon$ that contains each element of $\delta$ is called the \emph{coarse structure that is generated by $\delta$}. Then $\delta$ is called a \emph{subbase for $\varepsilon$}.
\item If $\varepsilon$ is a coarse structure and $\varepsilon'\s \varepsilon$ a subset such that $E\in \varepsilon$ implies there is some $E'\in \varepsilon'$ with $E\s E'$ then $\varepsilon'$ is called a \emph{base for $\varepsilon$.}
\end{itemize}
\end{rem} 

\begin{ex}
If $X$ is a set there are two trivial coarse structures on $X$:
\begin{enumerate}
\item the \emph{discrete coarse structure} consists of subset of the diagonal and finitely many off-diagonal points.
\item the \emph{maximal coarse structure} is generated by $X^2$. Note that in this case each subset of $X$ and in particular $X$ itself is bounded.
\end{enumerate}
\end{ex}

\begin{ex}
\label{ex:metricspaces+groups}
If $X$ is a metric space with metric $d$ then the \emph{bounded coarse 
structure} of $X$ consists of those subsets $E\s X^2$ for which
\[
\sup_{(x,y)\in E}d(x,y)<\infty.
\]
A coarse space $X$ is called metrizable if there is a metric $d$ that can be 
defined on it such that $X$ carries the bounded coarse structure associated to 
$d$. Note that by \cite[Theorem~2.55]{Roe2003} a coarse space is metrizable if 
and only if it has a countable base.  

Metric spaces which are of particular interest in the topic of coarse geometry 
are Riemannian manifolds and finitely generated groups. Let $G$ be a finitely 
generated group equipped with a finite generating set $S$. Then the word length 
$l:G\to \N$ according to $S$ assigns an element $g\in G$ with the minimal 
length of a word written in the alphabet $S\cup \ii S$ that represents $g$. Then 
the map
\f{
d:G\times G&\to G\\
(g,h)&\mapsto l(\ii gh)
}
is a metric on $G$. Note the metric depends on the generating set, the coarse 
structure associated to $d$ does not though.
\end{ex}

\begin{ex}
If $X$ is a paracompact and locally compact Hausdorff space and $\bar X$ a compactification of $X$ with boundary $\partial X$ then the topological coarse structure associated to the given compactification consists of subsets $E\s X^2$ such that
\[
\partial E\cap {\partial X}^2\ohne \Delta_{\partial X}=\emptyset.
\]
If the compactification is second countable then by~\cite[Example~2.53]{Roe2003} the topological coarse structure on $X$ is not metrizable.
\end{ex}

\subsection{Coarse Maps}

\begin{defn}\name{close}
Let $S$ be a set and let $X$ a be coarse space. Two maps $f,g:S\to X$ are called \emph{close} if
\[
 \{(f(s),g(s))|s\in S\}\s X^2
\]
is an entourage.
\end{defn}

\begin{defn}\name{maps}
Let $f:X\to Y$ be a map between coarse spaces. Then $f$ is called
\begin{itemize}
\item \emph{coarsely proper} if for every bounded set $B$ in $Y$ the inverse image $f^{-1}(B)$ is bounded in $X$;
\item \emph{coarsely uniform} if every entourage $E$ of $X$ is mapped by $\zz f=f\times f:X^2\to Y^2$ to an
entourage $\zz f(E)$ of $Y$;
\item  a \emph{coarse map} if it is both coarsely proper and coarsely uniform;
\item a \emph{coarse embedding} if $f$ is coarsely uniform and for every entourage $F\s Y^2$ the inverse image $\izp f{F}$ is an entourage.
\end{itemize}
\end{defn}

\begin{defn}\name{coarsely equivalent}
\begin{itemize}
 \item A coarse map $f:X\to Y$ between coarse spaces is a \emph{coarse equivalence} if there is a coarse map $g:Y\to X$ such that $f\circ g:Y\to Y$ is close to the identity on $Y$ and $g\circ f:X\to X$ is close to the identity on $X$.
\item two coarse spaces $X,Y$ are \emph{coarsely equivalent} if there is a coarse equivalence $f:X\to Y$.
\end{itemize}
\end{defn}

\begin{defn}
We denote by $\coarse$ the category with objects coarse spaces and morphisms 
coarse maps modulo close. Then coarse equivalences are the isomorphisms in the 
\emph{coarse category}.
\end{defn}

\section{Coentourages}
\label{sec:coents}
In this chapter coentourages are introduced. We study the dual characteristics of coentourages to entourages.

\subsection{Definition}
This is a special case of~\cite[Definition~5.3, p.~71]{Roe2003}:
\begin{defn}
\label{defn:coentourages}
Let $X$ be a coarse space. A subset $C\s X^2$ is called a \emph{coentourage} if for every entourage $E$ there is a bounded set $B$ such that
\[
 C\cap E\s B^2.
\]
 The set of coentourages in $X$ is called the \emph{cocoarse structure} of $X$.
\end{defn}

\begin{lem}
\label{lem:coentourages}
 The following properties hold:
 \begin{enumerate}
  \item Finite unions of coentourages are coentourages.
  \item Subsets of coentourages are coentourages.
  \item If $f:X\to Y$ is a coarse map between coarse spaces then for every coentourage $D\s Y^2$ the set $\izp f D$ is a coentourage.
 \end{enumerate}
\end{lem}
\begin{proof}
 \begin{enumerate}
  \item Let $C_1, C_2$ be coentourages. Then for every entourage $E$ there are bounded sets $B_1, B_2$ such that
  \f{
   (C_1\cup C_2)\cap E
   &=C_1\cap E \cup C_2\cap E\\
   &\s B_1\times B_1\cup B_2\times B_2\\
   &\s (B_1\cup B_2)^2.
  }
Now $B_1\cup B_2$ is bounded because $X$ is connected.
\item Let $C$ be a coentourage and $D\s C$ a subset. Then for every entourage $E$ there is some bounded set $B$ such that
\f{
D\cap E
&\s C\cap E\\
 &\s B^2.
}
\item This is actually a special case of~\cite[Lemma~5.4]{Roe2003}. For the convenience of the reader we include the proof anyway. 

Let $E$ be an entourage in $X$. Then there is some bounded set $B\s Y$ such that
\f{
 f^2(\izp f D\cap E)
 &\s D\cap f^2(E)\\
 &\s B^2.
}
But then
\f{
 \izp f D\cap E 
 &\s \iz f\circ \zz f(\izp f D\cap E)\\
 &\s \izp f {B^2}\\
 &= f^{-1}(B)^2.
}
 \end{enumerate}
\end{proof}

\begin{ex}
If $G$ is an infinite countable group then there is a canonical
coarse structure on $G$: A subset $E\s G^2$ is an entourage if the set
\[
 \{g^{-1}h: (g,h)\in E\}
\]
is finite. Note for finitely generated group this definition agrees with 
Example~\ref{ex:metricspaces+groups}. If $U,V\s G$ are two subsets of $G$ then 
$U\times V$ is a coentourage if $U\cap Vg$ is finite for every $g\in G$.
\end{ex}

\begin{ex}
 In the coarse space $\Z$ one can see three examples:
 \begin{itemize}
  \item the even quadrants are a coentourage: $\{(x,y): xy<0\}$.
  \item For $n\in \Z$ the set perpendicular to the diagonal with foot $(n,n)$ is a coentourage: 
$\{(n-x,n+x):x\in \Z\}$.
\item Here is another example: $\{(x,2x):x\in \Z\}$ is a coentourage.
 \end{itemize}
 \end{ex}
 
 \begin{ex}
  Look at the infinite dihedral group which is defined by
  \[
   D_\infty=\langle a,b:a^2=1,b^2=1 \rangle.
  \]
 In $D_\infty$ the set 
 \[
  \{(ab)^n,(ab)^na:n\in\N\}\times \{(ba)^n,(ba)^nb:n\in \N\}
 \]
is a coentourage.
 \end{ex}

\subsection{A Discussion/ Useful to know}

\begin{lem}
\label{lem:bounded}
 Let $X$ be a coarse space. Then for a subset $B\s X$ the set $B^2$ is a coentourage if and only if $B$ is bounded.
\end{lem}
\begin{proof}
  If $B$ is bounded then it is easy to see that $B^2$ is a coentourage.
   
  Conversely suppose $B^2$ is a coentourage. Then 
  \[
   \Delta_X\cap B^2\s B^2
  \]
and $B^2$ is the smallest squared subset of $X^2$ which contains
\[
 \{(b,b): b\in B\}
\]
which is $\Delta_X\cap B^2.$ Thus $B$ is bounded. 
\end{proof}

\begin{defn}\name{dual structure}
If $X$ is a coarse space let $\varepsilon$ and $\gamma$ be collections of 
subsets of $X^2$. Denote by $\beta$ the set of bounded sets. We say 
that \emph{$\varepsilon$ detects $\gamma$} if 
 \begin{enumerate}
 \item for every $D\not\in\gamma$ there is some $E\in \varepsilon$ such that 
$D\cap E\not\s B\times B$ for every $B\in \beta$.
\item a subset $D\s X\times X$ is contained in $\gamma$ if for every 
$E\in \varepsilon$ there is some $B\in\beta$ such that $D\cap E\s B\times B$ 
.
\end{enumerate}
Then we say that \emph{$\varepsilon$ is dual to $\gamma$} if $\varepsilon$ 
detects $\gamma$ and $\gamma$ detects $\varepsilon$. By definition the 
collection of entourages detects the collection of coentourages.
\end{defn}

\begin{prop}
\label{prop:metric_coarselynormal}
 Let $X$ be a coarse space with the bounded coarse structure of a metric 
space\footnote{In what follows coarse spaces with the bounded coarse structure 
of a metric space will be refered to as metric spaces.} then the coarse 
structure of $X$ is dual to the cocoarse structure of $X$. 
\end{prop}
\begin{proof}
 Let $F\s X^2$ be a subset which is not an entourage. Then for every entourage there is a point in $F$ that is not in $E$. Now we have a countable basis for the coarse structure:
\[
 E_1,E_2,\ldots,E_n,\ldots
\]
ordered by inclusion. Then there is also a sequence $(x_i,y_i)_i\s X^2$ with $(x_i,y_i)\not \in E_i$ and $(x_i,y_i)\in F$. Denote this set of points by $f$. Then for every $i$ the set 
\[
 E_i\cap f
\]
is a finite set of points, thus $f$ is a coentourage. But $F\cap f=f$ is not an entourage, specifically it is not contained in $B^2$ if $B$ is bounded.
\end{proof}

\begin{prop}
 Let $X$ be a paracompact and locally compact Hausdorff space. Let $\bar X$ be 
a compactification of $X$ and equip $X$ with the topological coarse structure 
associated to the given compactification. Then
\begin{enumerate}
 \item a subset $C\s X^2$ is a coentourage if $\bar C\cap \Delta_{\partial X}$ 
is empty.
 \item if $U,V$ are subsets of $X$ then $U\times V$ is a coentourage if 
$\partial U\cap \partial V=\emptyset$.
 \item the coarse structure of $X$ is dual to the cocoarse structure of $X$.
 \end{enumerate}
\end{prop}
\begin{proof}
 easy.
\end{proof}

\begin{lem}
\label{lem:coentent=coent}
Let $X$ be a coarse space. If $C\s X^2$ is a coentourage and $E\s X^2$ an entourage then $C\circ E$ and $E\circ C$ are coentourages.
\end{lem}
\begin{proof}
Let $F\s X^2$ be any entourage. Without loss of generality $E$ is symmetric and contains the diagonal. Now $C$ being a coentourage implies that there is some bounded set $B\s X$ such that
\[
C\cap E^{-1}\circ F\s B^2
\]
Then
\f{
E\circ C\cap F
&\s E\circ(C\cap E^{-1}\circ F)\\
&\s E\circ B^2\\
&\s (E[B]\cup B)^2
}
\end{proof}

If $X$ is a set a collection $\beta$ of subsets of $X$ is called a 
\emph{bornology} if
\begin{itemize}
      \item $X=\bigcup_{B\in\beta}B$;
      \item if $A\in \beta$ and $A'\s A$ then $A'\in\beta$.
      \item if $A,B\in \beta$ then $A\cup B\in\beta$.
\end{itemize}
      If $X$ is a metric space then the set of bounded subsets of $X$ are a 
bornology. We will use this property in Theorem~\ref{thm:detectcoarse}.

Now we are going to characterize coentourages axiomatically. 
\begin{thm}
\label{thm:detectcoarse}
 If $X$ is a set let $\gamma$ be a collection of subsets of $X^2$ such that
 \begin{enumerate}
  \item $\gamma$ is closed under taking subsets, finite unions and inverses;
  \item we say a subset $B\s X$ is bounded if $B\times X\in \gamma$ and require 
  \[
   X=\bigcup_{B\in \beta}B;
  \]
  \item for every $C\in \gamma$ there is some bounded set $B\s X$ such that
  \[
   C\cap \Delta_X\s B^2;
  \]
  \item If $E$ is detected by $\gamma$ and $C\in \gamma$ then $E\circ C\in\gamma$. 
 \end{enumerate}
Then $\gamma$ detects a coarse structure.
\end{thm}
\begin{proof}
 Denote by $\beta$ the collection of bounded sets of $X$. Note that by points 1 and 2 the system $\beta$ is a bornology. Now we show that $\gamma$ detects a coarse structure by checking the axioms in Definition~\ref{defn:coarsestructure}.
 \begin{enumerate}
  \item Point 3 guarantees that the diagonal is an entourage.
  \item That is because $\beta$ is a bornology.
  \item Same.
  \item By point 1 the inverse of an entourage is an entourage.
  \item Suppose $E,F\s X^2$ are detected by $\gamma$. 
Without loss of generality $E$ 
is symmetric and contains the diagonal. Then there is some bounded set $B$ such that
\[
 F\cap \ii E\circ C\s B^2.
\]
But then
\f{
E\circ F\cap C&\s E\circ (F\cap \ii E\circ C)\\
&\s E\circ B^2\\
&\s (E[B]\cup B)^2
}
and that is bounded because of the first point.
\item this works because of point 2.
 \end{enumerate}
\end{proof}

\begin{defn}\name{coarsely disjoint}
 If $A,B\s X$ are subsets of a coarse space then $A$ is called \emph{coarsely disjoint to} $B$ if 
 \[
 A\times B\s X^2
 \]
 is a coentourage. Being coarsely disjoint is a relation on subsets of $X$.
\end{defn}

\subsection{On Maps}

Note that in this chapter every coarse space is assumed to have the property 
that the coarse structure is dual to the cocoarse structure.

\begin{lem}
Two coarse maps $f,g:X\to Y$ are close if and only if 
for every coentourage $D\s Y^2$ the set $(f\times g)^{-1}(D)$ is  
a coentourage.
\end{lem}
\begin{proof}
  Denote by $\beta$ the collection of bounded sets. Suppose $f,g$ are close. Let $C\s Y^2$ be a coentourage and $E\s X^2$ an entourage. Set 
  \[
   S=\ii {(f\times g)}(C)\cap E.
  \]
  Then there is some bounded set $B$ such that
  \f{
  (f\times g)(S)
  &=(f\times g)\circ(\ii{(f\times g)}(C)\cap E)\\
  &\s (f\times g)\circ \ii{(f\times g)}(C)\cap (f\times g)(E)\\
  &\s C\cap (f\times g)(E)\\
  &\s B^2.
  }
  But $f$ and $g$ are coarsely proper thus
  \f{
   S
   &\s (\ii f \times \ii g)\circ (f\times g)(S)\\
   &\s \ii f(B)\times \ii g(B)
  }
  is in $\beta^2$.

  Now for the reverse direction: Let $C\s Y^2$ be a coentourage. There is some bounded set $B\s X^2$ such that
\[
 \Delta_X\cap \ii{(f\times g)}(C)\s B^2.
\]
Then
\f{
(f\times g)(\Delta_X)\cap C
&=(f\times g)(\Delta_X)\cap(f\times g)\circ \ii {(f\times g)}(C)\\
&=(f\times g)(\Delta_X\cap\ii{(f\times g)}(C))\\
&\s (f\times g)(B^2).
}
But $f,g$ are coarsely uniform thus $(f\times g)(B^2)\in \beta^2$. 
\end{proof}

\begin{prop}
 A map $f:X\to Y$ between coarse spaces is coarse if and only if
 \begin{itemize}
  \item for every bounded set $B\s X$ the image $f(B)$ is bounded in $Y$ 
  \item and for every coentourage $C\s Y^2$ the reverse image $\izp f C$ is a coentourage in $X$
 \end{itemize}
\end{prop}
\begin{proof}
  Suppose $f$ is coarse. By Lemma~\ref{lem:coentourages} point 3 the second point holds and by coarsely uniformness the first point holds.

  Suppose the above holds. Let $E\s X^2$ be an entourage. For every coentourage $D\s Y^2$ there is some bounded set $B$ such that
\[
 E\cap \izp f D\s B^2.
\]
Then
\f{
\zz f(E)\cap D&=\zz f(E)\cap \zz f\circ \izp f D\\
&=\zz f(E\cap \izp f D)\\
&\s f(B)^2.
}
Because of point 1 we have $\zz f(B)\in \beta$. By point 2 the reverse image of every bounded set is bounded.
\end{proof}

\begin{defn}
\label{defn:coarselysurjective}
 A map $f:X\to Y$ between coarse spaces is called \emph{coarsely surjective} if one of the following equivalent conditions applies:
 \begin{itemize}
\item There is an entourage $E\s Y^2$ such that $E[\im f]=Y$.
 \item there is a map $r:Y\to \im f$ such that
 \[
  \{(y,r(y)):y\in Y\} 
 \]
is an entourage in $Y$.
\item The inclusion $\im f\to Y$ is a coarse equivalence.
\end{itemize}
We will refer to the above map $r$ as the retract of $Y$ to $\im f$. Note that it is a coarse equivalence.
\end{defn}

\begin{lem}
 Every coarse equivalence is coarsely surjective.
\end{lem}
\begin{proof}
 Let $f:X\to Y$ be a coarse equivalence and $g:Y\to X$ its inverse. Then $f\circ g:Y\to \im f$ is the retract of Definition~\ref{defn:coarselysurjective}.
\end{proof}

\begin{lem}
 Coarsely surjective coarse maps are epimorphisms in the category of coarse spaces and coarse maps modulo close.
\end{lem}
\begin{proof}
Suppose $f:X\to Y$ is a coarsely surjective coarse map between coarse spaces. Then there is an entourage $E\s Y^2$ such that $E[\im f]=Y$. We show $f$ is an epimorphism. Let $g_1,g_2:Y\to Z$ be two coarse maps to a coarse space such that $g_1\circ f,g_2\circ f$ are close. Then the set 
\[
H:=g_1\circ f\times g_2\circ f(\Delta_X)
\]
is an entourage. Then
\[
g_1\times g_2(\Delta_Y)\s \zzp {g_1}E\circ H\circ \zzp {g_2} {\ii E}
\]
is an entourage. Thus $g_1,g_2$ are close.
\end{proof}

\begin{defn}
\label{defn:coarselyinjective}
 A map $f:X\to Y$ between coarse spaces is called \emph{coarsely injective}\footnote{Note that every coarsely injective coarse map is called a coarse embedding. Although term 'coarse embedding' is in general use and describes the notion more appropriately we will use the former term 'coarsely injective' because adjectives are easier to handle.} if for every entourage $E\s Y^2$ the set $\izp f E$ is an entourage. 
\end{defn}

\begin{rem}
 Let $f:X\to Y$ be a map between coarse spaces. If $f$ is coarsely injective 
and maps bounded sets to bounded sets then $\zzp f C$ is a coentourage for every 
coentourage $C\s X^2$.

If on the other hand $\zz f$ maps coentourages to coentourages, the space $X$ 
is metric and $f$ is coarsely proper then $f$ is coarsely injective.
\end{rem}
\begin{proof}
Suppose $f$ is coarsely injective and maps bounded sets to bunded sets. Let $C\s X^2$ be a coentourage and $E\s Y^2$ be an entourage. Then there exists a bounded set $B\s X$ such that $C\cap \izp f E\s B^2$. Then
\f{
\zzp f {B^2}&\z \zzp f {C\cap\izp f E}\\
&=\zzp f C \cap E.
}
Since $E$ was an arbitrary entourage this implies that $\zzp f C$ is a coentourage in $Y$.

Now suppose $\zz f$ maps coentourages to coentourages, $X$ is a metric space and $f$ is coarsely proper. Let $E\s Y^2$ be an entourage and $C\s X^2$ be a coentourage. Then there exists a bounded set $B\s Y$ such that $\zzp f C \cap E\s B^2$. Then
\f{
\izp f {B^2}
&\z \izp f {\zzp f C \cap E}\\
&=\iz f\circ \zzp f C\cap \izp f E\\
&\z C\cap \izp f E.
} 
Since $C$ was an arbitrary coentourage the set $\izp f E$ is an entourage by 
Proposition~\ref{prop:metric_coarselynormal}.
\end{proof}

\begin{lem}
\label{lem:coarseequivinjective}
 Let $f:X\to Y$ be a coarse equivalence. Then $f$ is coarsely injective.
\end{lem}
\begin{proof}
 Let $g:Y\to X$ be a coarse inverse of $f$. Then there is an entourage 
 \[
  F=\{(g\circ f(x),x):x\in X\}
 \]
 in $X$.
But then $g\circ f$ is coarsely injective because for every coentourage $C\s X^2$ we have
\[
 \zz{g\circ f}(C)\s F\circ C\circ F^{-1}
\]
and $ F\circ C\circ F^{-1}$ is again a coentourage by Lemma~\ref{lem:coentent=coent}. But
\[
 \zz f(C)\s \iz g\circ \zz g\circ \zz f(C)
\]
is a coentourage, thus $f$ is coarsely injective.
\end{proof}

\begin{lem}
 Coarsely injective coarse maps are monomorphisms in the category of coarse spaces and coarse maps modulo closeness.
\end{lem}
\begin{proof}
Suppose $f:X\to Y$ is a coarsely injective coarse map between coarse spaces. We show $f$ is a monomorphism. Let $g_1,g_2:Z\to X$ be two coarse maps such that $f\circ g_1,f\circ g_2:Z\to Y$ are close. Then
\[
H:=f\circ g_1\times f\circ g_2(\Delta_Z)
\]
is an entourage. Now
\[
g_1\times g_2(\Delta_Z)\s \izp f H
\]
is an entourage. Thus $g_1,g_2$ are close.
\end{proof}

\begin{rem}
 Every coarse map can be factored into an epimorphism followed by a monomorphism.
\end{rem}

\begin{prop}
 If a coarse map $f:X\to Y$ is coarsely surjective and coarsely injective then $f$ is a coarse equivalence.
\end{prop}
\begin{proof}
 We just need to construct the coarse inverse. Note that the map $r:Y\to \im f$ from the second point of Definition~\ref{defn:coarselysurjective} is a coarse equivalence which is surjective. Without loss of generality we can replace $f$ by $\hat f=r\circ f$. Now define $g:\im f\to X$ by mapping $y\in \im f$ to some point in $\hat f^{-1}(y)$ where the choice is not important. Now we show:
\begin{enumerate}
 \item $g$ is a coarse map: Let $E\s (\im f)^2$ be an entourage. Then
 \[
  \zz g(E)\s \iz f(E)
 \]
 is an entourage. And if $B\s X$ is bounded then
 \[
  \iip g B \s f(B)
 \]
 is bounded.
 \item $\hat f\circ g=id_{\im f}$
 \item $g$ is coarsely injective: Let $D\s (\im f)^2$ be a coentourage. Then
 \[
  \zz g(D)\s \izp f D
 \]
is a coentourage.
\item $g\circ \hat f\sim id_X$: we have $g\circ \hat f: X\to \im g$ is coarsely injective and thus the retract of Definition~\ref{defn:coarselysurjective} with coarse inverse the inclusion $i:\im g \to X$. But
\[
 g\circ \hat f\circ i=id_{\im g}.
\]
\end{enumerate}
\end{proof}

\section{Limits and Colimits}
\label{sec:lac}
The category $\topology$ of topological spaces is both complete and cocomplete. In fact the forgetful functor $\topology\to \sets$ preserves all limits and colimits that is because it has both a right and left adjoint. We do something similar for coarse spaces.

Note that the following notions generalize the existing notions of product and disjoint union of coarse spaces.

\subsection{The Forgetful Functor}

 \begin{defn}
 Denote the category of connected coarse spaces and coarsely uniform maps between them by $\dcoarse$.
\end{defn}

\begin{thm}
 The forgetful functor $\eta:\dcoarse\to \sets$ preserves all limits and colimits.
\end{thm}
\begin{proof}
There is a functor $\delta:\sets\to \dcoarse$ that sends a set $X$ to the 
coarse space $X$ with the discrete coarse structure\footnote{in which every 
entourage is the union of a subset of the diagonal and finitely many 
off-diagonal points}. Then every map of sets induces a coarsely uniform map.

There is a functor $\alpha:\sets\to \dcoarse$ which sends a set $X$ to the 
coarse space $X$ with the maximal coarse structure. Again every map of sets 
induces a coarsely uniform map.

Let $X$ be a set and $Y$ a coarse space. Then
 \[
  Hom_\sets(X,\eta Y)=Hom_\dcoarse(\delta X, Y)
 \]
and
\[
 Hom_\sets(\eta Y, X)=Hom_\dcoarse(Y,\alpha X)
\]
Thus the forgetful functor is right adjoint to $\delta$ and left adjoint to $\alpha$.

An application of the~\cite[Adjoints and Limits Theorem~2.6.10]{Weibel1994} 
gives the result.
\end{proof}

\subsection{Limits}

The following definition is a generalization of~\cite[Definition~1.21]{Grave2006}:
\begin{defn}
\label{defn:pullbackcoarsestructure}
Let $X$ be a set and $f_i:X\to Y_i$ a family of maps to coarse spaces. The 
\emph{pullback coarse structure of $(f_i)_i$} on $X$ is generated by $\bigcap_i 
\izp {f_i}{E_i}$ for $E_i\s Y_i$  an entourage for every $i$. That is, a subset 
$E\s X^2$ is an entourage if for every $i$ the set $\zzp{f_i}{E}$ is an 
entourage in $Y_i$.
\end{defn}

\begin{lem}
\label{lem:pullbackclose}
 The pullback coarse structure is indeed a coarse structure; the maps $f_i:X\to 
Y_i$ are coarsely uniform.
\end{lem}
\begin{proof}
We check the axioms of a coarse structure:
 \begin{enumerate}
  \item $\Delta_X\s \izp {f_i} {\Delta_{Y_i}}$ for every $i$.
  \item If $E\s X\times X$ is an entourage and $F\s E$ a subset then $\zzp{f_i} 
F\s \zzp {f_i} E$ is an entourage in $Y_i$ for every $i$.
  \item if $E_1, E_2$ are entourages in $X$ then for every $i$ there are entourages $F_1,F_2\s Y_i^2$ such that $E_1\s \izp{f_i}{F_1}$ and $E_2\s\izp{f_i}{F_2}$. But then 
  \f{
  E_1\cup E_2&\s \izp {f_i} {F_1} \cup \izp {f_i} {F_2}\\
  &=\izp{f_i}{F_1\cup F_2}
  }
  \item if $E$ is an entourage in $X$ then for every $i$ there is an entourage $F$ in $Y_i$ such that $E\s \izp{f_i}{F}$. But then 
\[
E^{-1} \s \izp {f_i} {F^{-1}}
\]
  \item If $E_1,E_2$ are as above then 
  \[
  E_1\circ E_2\s \izp {f_i} {F_1\circ F_2}
  \]
  \item If $(x,y)\in X$ then for every $i$ 
  \[
   \zzp{f_i}{x,y}=(f_i(x),f_i(y))
  \]
 is an entourage.
 \end{enumerate}
\end{proof}

\begin{rem}
Note that it would be ideal if the pullback coarse structure is well-defined up 
to coarse equivalence and   if there is a universal property. We cannot use 
naively the limit in $\sets$ and equip it with the pullback coarse structure as 
the following example shows:

Denote by $\phi:\Z\to \Z$ the map that maps $i\mapsto 2i$ and by $\psi:\Z\to \Z$ the map that maps $i\mapsto 2i+1$. then both $\phi,\psi$ are isomorphisms in the coarse category. The pullback of
\[
\xymatrix{
&
\Z\ar[d]^\phi\\
\Z\ar[r]_\psi
& \Z
}
\]
is $\emptyset$ in $\sets$ but should be an isomorphism if the diagram is supposed to be a pullback diagram in $\coarse$. 
\end{rem}

 \begin{prop}
 Let $X$ have the pullback coarse structure of $(f_i:X\to Y_i)_i$. A subset $C\s X^2$ is a coentourage if for every $i$ the set $\zzp{f_i}{C}$ is a coentourage in $Y_i$. Note that the converse does not hold in general.
\end{prop}
\begin{proof}
 Let $C\s X^2$ have the above property. If $F\s X^2$ is a subset such that
 \[
  S=C\cap F
 \]
is not bounded then there is some $i$ such that $\zzp{f_i}{S}$ is not bounded. Then
\[
 \zz{f_i}(C)\cap \zz{f_i}(F)\z \zz{f_i}(C\cap F)
\]
is not bounded but $\zzp{f_i}{C}$ is a coentourage in $Y_i$. Thus $\zzp{f_i}{F}$ is not an entourage in $Y_i$, thus $F$ does not belong to the pullback coarse structure on $X$. Thus $C$ is detected by the pullback coarse structure.
\end{proof}

\begin{ex}\name{Product}
The pullback coarse structure on products agrees with~\cite[Definition~1.32]{Grave2006}: If $X,Y$ are coarse spaces the product $X\times Y$ has the pullback coarse structure of the two projection maps $p_1,p_2$: 
\begin{itemize}
 \item A subset $E\s (X\times Y)^2$ is an entourage if and only if $\zz{p_1}(E)$ is an entourage in $X$ and $\zz{p_2}(E)$ is an entourage in $Y$.
\item A subset $C\s (X\times Y)^2$ is a coentourage if and only if $\zz{p_1}(C)$ is a coentourage in $X$ and $\zz{p_2}(C)$ is a coentourage in $Y$.
\end{itemize}
\end{ex}

\subsection{Colimits}

\begin{prop} If $f_i:Y_i\to X$ is a finite family of injective maps from coarse spaces then the subsets
 \[
  \zz{f_i}(E_i)
 \]
 for $i$ an index and $E_i\s Y_i^2$ an entourage are a subbase for a coarse structure; the maps $f_i:Y_i\to X$ are coarse maps.
\end{prop}
\begin{proof}
 Suppose $E_i\s Y_i^2$ is an entourage. Let $C\s X^2$ be an element of the pushout cocoarse structure. Denote
 \[
  S=\zz{f_i}(E_i)\cap C.
 \]
Then 
\f{
\izp {f_i} S
&=\iz {f_i}\circ \zz{f_i}(E_i)\cap \izp{f_i} C\\
&=E_i\cap \izp {f_i} C
}
implies that $\zz{f_i}(E_i)$ is an entourage.

Now $E\s X^2$ is an entourage if for every $i$ 
\[
 E\cap (\im f_i)^2
\]
is an entourage and if $E\cap (\bigcup_i (\im f_i)^2)^c$ is bounded.

We show that this is indeed a coarse structure by checking the axioms of Definition~\ref{defn:coarsestructure}:
\begin{enumerate}
  \item We show the diagonal in $X$ is an entourage. Let $C\s X^2$ be a subset such that 
  \[
  \izp {f_i} C\s Y_i^2
  \]
  is a coentourage. Denote
\[
 S=\Delta_X\cap C.
\]
Then
\f{
\izp {f_i} {\Delta_X\cap C}
&=\izp {f_i} {\Delta_X}\cap \izp {f_i} C\\
&=\Delta_{Y_i}\cap \izp {f_i} C\\
&\s B_i^2
}
is bounded.
  \item easy
  \item easy
  \item easy
  \item If $E_1,E_2\s X^2$ have the property that for every element $C\s X^2$ of the pushout cocoarse structure and every $i$: 
\[
 \izp {f_i} {E_1}\cap \izp {f_i} C
\]
and
\[
 \izp {f_i} {E_2}\cap \izp {f_i} C
\]
are bounded in $Y_i$ we want to show that $E_1\circ E_2$ has the same property. Now without loss of generality we can
assume that there are $ij$ such that $E_1\s (\im f_i)^2$ and $E_2\s (\im f_j)^2$ the other cases being trivial or they
can be reduced to that case. Then 
\[
 E_1\circ (E_2\cap (\im f_i)^2)\s (\im f_i)^2
\]
and
\[
 (E_1\cap (\im f_j)^2)\circ E_2\s (\im f_j)^2
\]
are entourages and the other cases are empty.
  \item If $(x_1,x_2)\in X^2$ then for every $i$
  \[
   \izp {f_i}{x_1,x_2}
  \]
 is either one point or the empty set in $Y_i$, both are entourages.
 \end{enumerate}
\end{proof}

\begin{defn}
\label{defn:colimit}
 Let $X$ be a set and $f_i:Y_i\to X$ a finite family of injective maps from coarse spaces. Then define the \emph{pushout cocoarse structure} on $X$ to be those subsets $C$ of $X^2$ such that for every $i$ the set
\[
 \izp {f_i} C\s Y_i^2
\]
is a coentourage.
\end{defn}

\begin{ex}
 Let $A,B$ be coarse spaces and $A\sqcup B$ their disjoint union. The cocoarse structure and the coarse structure of $A\sqcup B$ look like this: 
\begin{itemize}
 \item A subset $D\s (A\sqcup B)^2$ is a coentourage if $D\cap A^2$ is a coentourage in $A$ and $D\cap B^2$ is a coentourage in $B$.
 \item A subset $E\s (A\sqcup B)^2$ is an entourage if $E\cap A^2$ is an entourage of $A$ and $E\cap B^2$ is an entourage of $B$ and $E\cap (A\times B\cup B\times A)$ is contained in $S\times T\cup T\times S$ where $S$ is bounded in $A$ and $T$ is bounded in $B$. This definition actually agrees with~\cite[Definition~2.12, p.~277]{Mitchener2001}. 
\end{itemize}
\end{ex}

\begin{ex}
 Let $G$ be a countable group that acts on a set $X$. We require that for every $x,y\in X$ the set 
 \[
  \{g\in G:g.x=y\}
 \]
 is finite. Then the pushout cocoarse structure of the orbit maps
 \f{
 i_x:G &\to X\\
 g &\mapsto g.x
 }
 for $x\in X$ is dual to the minimal connected $G-$invariant coarse structure of~\cite[Example~2.13]{Roe2003}.
\end{ex}
\begin{proof}
 Note that by the above requirement a subset $B\s X$ is bounded if and only if it is finite. Fix an element $x\in X$ and denote by $X'\s X$ the orbit of $x$.

 For every $C\s G^2$ coentourage
 \[
  E\cap i_x^2(C)
 \]
being bounded implies that
\[
 \izp {i_x} E\cap C\s \izp {i_x}{E\cap \zz{i_x}(C)}
\]
is bounded. Thus if $E\s X^2$ is an entourage then $\izp {i_x} E$ is an entourage.

If $\izp {i_x} E$ is an entourage then $E=\zz{i_x}\circ \izp {i_x} E$. For every $C\s G^2$ coentourage
\[
 \izp {i_x} E\cap C
\]
being bounded implies that
\[
 E\cap \zz{i_x}(C)
\]
is bounded. Thus $E$ is an entourage.

The $\zz{i_x}(E)$ for $E\s G^2$ an entourage are a coarse structure on $X'$ because $i_x$ is surjective on $X'$.

If $x,y$ are in the same orbit $X'$ then $i_x,i_y$ induce the same coarse structure on $X'$. 
\end{proof}

\section{Coarse Cohomology with twisted Coefficients}
\label{sec:ct}

We define a Grothendieck topology on coarse spaces and describe cohomology with twisted coefficients on coarse spaces and coarse maps. We have a notion of Mayer-Vietoris and a notion of relative cohomology.

\subsection{Coarse Covers}
\begin{defn}
\label{defn:coarsecover}
 Let $X$ be a coarse space and let $(U_i)_i$ be a finite family of subspaces of $X$. It is said to \emph{coarsely cover} $X$ if the complement of
\[
 \bigcup_i U_i^2
\]
is a coentourage.
\end{defn}

\begin{ex}
 The coarse space $\Z$ is coarsely covered by $\Z_-$ and $\Z_+$. An example for a decomposition that does not coarsely cover $\Z$ is $\{x\in\Z: x\mbox{ is even}\}\cup\{x\in\Z: x\mbox{ is odd}\}$.
\end{ex}

\begin{rem}
 The finiteness condition is important, otherwise $(\{x,y\})_{x,y\in X}$ would coarsely cover $X$, but if $X$ is not bounded we don't want $X$ to be covered by bounded sets only. 
\end{rem}

\begin{lem}
\label{lem:boundedcover}
 A nonbounded coarse space $X$ is coarsely covered by one element $U$ if and only if $X\ohne U$ is bounded.
\end{lem}
\begin{proof}
By definition $U$ coarsely covers $X$ if and only if $(U^2)^c$ is a coentourage; now $(U^c)^2\s (U^2)^c$ thus $U^c$ is bounded by Lemma~\ref{lem:bounded}.

Conversely, if $U^c$ is bounded then 
\[
(U^2)^c=X\times U^c\cup U^c\times X
\]
is a coentourage, thus $U$ coarsely covers $X$.
\end{proof}

\begin{rem}
 If $X$ is coarsely covered by $(U_i)_i$ and they cover $X$ (as sets) then it is the colimit (see Definition \ref{defn:colimit}) of them:
 \[
  X\cong\bigcup_i U_i
 \]
as a coarse space.
\end{rem}

This is going to be useful later:
\begin{prop}
 A finite family of subspaces $(U_i)_i$ coarsely covers a metric space $X$ if and only if for every entourage $E\s X^2$ the set
 \[
  E[U_1^c]\cap \ldots \cap E[U_n^c]
 \]
 is bounded.
\end{prop}
\begin{rem}
 This appeared already in~\cite[Definition~2.1]{Dranishnikov1998}; wherein $U_1^c,\ldots,U_n^c$ is a finite system of subsets of $X$ that diverges.
\end{rem}
\begin{proof}
We proceed by induction on the number $i$ of pieces in the cover.
 
 If there is one piece $U_1$, then by Lemma~\ref{lem:boundedcover} one subset $U_1\s X$ coarsely covers $X$ if and only if $U_1^c$ is bounded. By this and Lemma~\ref{lem:boundedsets} for every entourage $E\s X^2$ the set $E[U_1^c]$ is bounded. 
 
 Conversely if $E[U_1^c]$ is bounded for every entourage $E\s X^2$ then $U_1^c$ itself is bounded which implies that $U_1$ coarsely covers $X$.
 
 Consider next the case of two subsets $U_1,U_2$.  We first claim that they form a coarse cover if and only if $U_1^c \times U_2^c$ is a coentourage. Indeed $X^2\ohne (U_1^2\cup U_2^2) = U_1^c \times U_2^c \cup U_2^c \times U_1^c$, so $X^2\ohne (U_1^2\cup U_2^2)$ is a coentourage if and only if both of $U_1^c \times U_2^c$ and $U_2^c \times U_1^c$ are coentourages. Let $E\s X^2$ be an entourage. Now by Lemma~\ref{lem:coentent=coent} this implies that $U_1^c\times E[U_2^c]$ is a coentourage, namely we have that the set $E[U_1^c]\cap E[U_2^c]$ is bounded.
 
 Conversely from the assumption that $E[U_1^c] \cap E[U_2^c]$ is bounded for every entourage $E\s X^2$, we deduce $E[U_1^c]\cap U_2^c$ is a bounded set. This implies that $U_1^c\times U_2^c$ is a coentourage.

 Now we consider the inductive step. Suppose $n\ge 3$. Subsets $U_1,\ldots,U_n$ form a coarse cover if and only if for every $i<j$ the sets $\{U_i\cup U_j\}\cup\{U_k:k\not=i,j\}$ form a coarse cover of $X$. Let $E$ be an entourage. By the induction hypothesis
 \[
  E[(U_i\cup U_j)^c]\cap E[U_1^c]\cap\cdots\cap \widehat {E[U_i^c]}\cap\cdots\cap \widehat{E[U_j^c]}\cap\cdots \cap E[U_n^c]
 \]
is bounded for every $i<j$. Since $E[U_i^c]\z E[U_i^c\cap U_j^c]$ we obtain
\[
 E[U_1^c]\cap \cdots\cap E[U_n^c]=\bigcap_{i<j}E[(U_i\cup U_j)^c]\cap E[U_1^c]\cap\cdots\cap \widehat {E[U_i^c]}\cap\cdots\cap \widehat{E[U_j^c]}\cap\cdots \cap E[U_n^c]
\]
is bounded.

Conversely let $U_1,\ldots,U_n\s X$ be subsets with $E[U_1^c]\cap \cdots\cap E[U_n^c]$ bounded. Then since $E[U_i^c\cap U_j^c]\s E[U_i^c]\cap E[U_j^c]$ we obtain that
\[
 E[(U_i\cup U_j)^c]\cap E[U_1^c]\cap\cdots\cap \widehat {E[U_i^c]}\cap\cdots\cap \widehat{E[U_j^c]}\cap\cdots \cap E[U_n^c]
\]
is bounded for every $i<j$. By induction hypothesis $\{U_i\cup U_j\}\cup\{U_k:k\not=i,j\}$ is a coarse cover for every $i<j$. Thus $U_1,\ldots,U_n$ is a coarse cover.

\end{proof}

\begin{prop}
\label{prop:shiftcovers}
If $r:X\to Y$ is a surjective coarse equivalence then $(V_i)_i$ is a coarse cover of $Y$ if and only if $(\iip r{V_i})_i$ is a coarse cover of $X$.
\end{prop}
\begin{proof}
Suppose $(V_i)_i$ is a coarse cover of $X$. then $(\bigcup_i V_i^2)^c$ is a coentourage in $Y$ thus
\[
\bigcup_i \iip f {V_i}^c=\izp f {(\bigcup_i V_i)^c}
\]
is a coentourage. Thus $(\iip f{V_i})_i$ is a coarse cover of $X$.

Conversely suppose $(\iip f{V_i})_i$ is a coarse cover of $X$ then 
\[
(\bigcup_i V_i)^c=\zz f\circ\izp f {(\bigcup_i V_i)^c}
\]
is a coentourage in $Y$. 
\end{proof}

\subsection{The Coarse Site}
\begin{rem}
  In what follows we define a Grothendieck topology on the category of subsets 
of a coarse space $X$. What we call a Grothendieck topology is sometimes called 
a Grothendieck pretopology. We stick to the terminology of~\cite{Artin1962}. If 
$\catc$ is a category a Grothendieck topology $T$ on $\catc$ consists of
\begin{itemize}
 \item the underlying category $Cat(T)=\catc$
 \item the set of coverings $Cov(T)$ which consists of families of morphisms in $\catc$ with a common codomain. We write
\[
 \{U_i\to U\}_i
\]
where $i$ stands for the index. They comply with the following rules:
\begin{enumerate}
 \item Every isomorphism is a covering.
 \item \emph{Local character:} If $\{U_i\to U\}_i$ is a covering and for every $i$ the family $\{V_{ij}\to U_i\}_j$ is a covering then the composition
 \[
 \{V_{ij}\to U_i\to U\}_{ij}
 \]
 is a covering.
 \item \emph{Stability under base change:} For every object $U\in Cat(T)$, morphism $V\to U$ and covering $\{U_i\to U\}_i$ all fibre products $U_i\times_U V$ exist and the family
 \[
 \{U_i\times_U V\to V\}
 \]
 is a covering.
\end{enumerate}
\end{itemize}
In the course of this paper we will mostly (but not always) apply the theory 
of Grothendieck topologies as portrayed in~\cite[parts I,II]{Tamme1994}.
 \end{rem}

\begin{defn}
\label{defn:grothendiecktopology}
 To a coarse space $X$ is associated a Grothendieck topology $\coarsetopology X$ where the underlying category of $\coarsetopology X$ consists of subsets of $X$, there is an arrow $U\to V$ if $U\s V$. A finite family $(U_i)_i$ covers $U$ if they coarsely cover $U$, that is, if
\[
 U^2\cap (\bigcup_i U_i^2)^c
\]
is a coentourage in $X$.
\end{defn}
\begin{lem}
 The construction $\coarsetopology X$, is indeed a Grothendieck topology.
\end{lem}
\begin{proof}
We check the axioms for a Grothendieck topology:
 \begin{enumerate}
  \item if $U\s X$ is a subset the identity $\{U\to U\}$ is a covering
  \item Let $\cover {U_i} U i$ be a covering and suppose for every $i $ there is a covering $\cover {U_{ij}} {U_i} j$, then:
\f{
U^2\cap(\bigcup_{ij} U_{ij}^2)^c
&=U^2\cap\bigcap_i\bigcap_jU_{ij}^{2c}\\
&=\bigcap_i(U^2\cap\bigcap_jU_{ij}^{2c})\\
&=\bigcap_i[(U^2\cap U_i^2\cap\bigcap_jU_{ij}^{2c})\cup (U^2\cap U_i^{2c}\cap\bigcap_jU_{ij}^{2c})]\\
&\s \bigcap_i [(U_i^2\cap\bigcap_jU_{ij}^{2c}) \cup (U^2\cap U_i^{2c})]\\
&\s\bigcup_i (U_i^2\cap \bigcap_j U_{ij}^{2c})\cup \bigcap_i (U^2\cap U_i^{2c})\\
&=\bigcup_i [U_i^2\cap(\bigcup_{j} U_{ij}^2)^c]\cup [U^2\cap(\bigcup_{ij} U_i^2)^c];
}
Therefore $U^2\cap (\bigcup_{ij}U^2_{i,j})^c$ is a finite union of coentourages, since the index set is finite; so it is a coentourage by Lemma~\ref{lem:coentourages}.
\item Let $\cover {U_i} U i$ be a covering and let $V\s U$ be an inclusion. Then
\f{
 V^2\cap (\bigcup_i (V \cap U_i)^2)^c 
 &=V^2\cap \bigcap_i (V\cap U_i)^{2c}\\
 &=V^2\cap \bigcap_i (U_i^{2c}\cup V^{2c})\\
 &=V^2\cap \bigcap_i U_i^{2c}\\
 &=V^2\cap (\bigcup_i U_i^2)^c\\
 &\s U^2\cap (\bigcup_i U_i^2)^c
}
So $\{V\cap U_i\to V\}_i$ is a covering of $X_{ct}$.
 \end{enumerate}
\end{proof}

\begin{rem}
If $T,T'$ are two Grothendieck topologies a functor $f:Cat(T)\to Cat(T')$ is called a \emph{morphism of topologies} if 
\begin{enumerate}
\item if $\{\varphi_i:U_i\to U\}_i$ is a covering in $T$ then $\{f(\varphi_i):f(U_i)\to f(U)\}_i$ is a covering in $T'$.
\item if $\{U_i\to U\}_i\in Cov(T)$ and $V\to U$ a morphism in $Cat(T)$ then the canonical morphism
\[
f(U_i\times_U V)\to f(U_i)\times_{f(U)}f(V)
\]
is an isomorphism for every $i$.
\end{enumerate}
\end{rem}

\begin{defn}
 Let $f:X\to Y$ be a coarse map between coarse spaces. Then we define a functor 
 \f{
 \ii f:Cat(\coarsetopology Y)&\to Cat(\coarsetopology X)\\
 U&\mapsto \iip f U
 }
\end{defn}

\begin{lem}
\label{lem:morphismoftopologies}
 The functor $\ii f$ induces a morphism of Grothendieck topologies $\ii f:\coarsetopology Y\to \coarsetopology X$.
\end{lem}
\begin{proof}
We check the axioms for a morphism of Grothendieck topologies:
 \begin{enumerate}
  \item Let $\cover {U_i} U i$ be a covering in $Y$. Then
\[
 \iip f U^2\cap(\bigcup_i \iip f {U_i}^2)^c=\izp f{U^2\cap(\bigcup_i U_i^2)^c}
\]
is a coentourage. Thus $\cover {\iip f {U_i}} {\iip f U} i$ is a covering in $X$.
\item for every $U,V$ subsets of $X$ we have 
\[
 \iip f{U\cap V}=\iip f U \cap \iip f V
\]
 \end{enumerate}
\end{proof}

\begin{rem}
Let $T$ be a Grothendieck topology.
\begin{itemize}
\item A \emph{presheaf on $T$ with values in $\catc$} is defined as a contravariant functor $\sheaff:Cat(T)\to \catc$. 
\item A \emph{morphism $\eta:\sheaff\to \sheafg$ of presheaves with values in $\catc$} is a natural transformation of contravariant functors.
\item A presheaf is a \emph{sheaf on $T$} if for every covering $\{U_i\to U\}\in Cov(T)$ the diagram 
\[
\sheaff(U)\to \prod_i \sheaff(U_i) \rightrightarrows \prod_{ij} \sheaff(U_i\times_U U_j)
\]
is an equalizer diagram in $\catc$. Exactness at $\sheaff(U)$ means that the first arrow $s\mapsto (s|_{U_i})_i$ is injective (\emph{global axiom}) and exactness at $\prod_i\sheaff(U_i)$ means that the image of the first arrow is equal to the kernel of the double arrow, hence consists of all $(s_i)_i$ such that $s_i|_{U_j}=s_j|_{U_i}$ (\emph{gluing axiom}).
\item A \emph{morphism of sheaves} is a morphism of the underlying presheaves.  
\end{itemize}
\end{rem}

\begin{ex}
\label{ex:bounded}
 Let $B$ be a space with the indiscrete (maximal) coarse structure. Then $B$ is already covered by the empty covering. But then the equalizer diagram for that covering is 
\[
 \sheaff(B)\to \prod_\emptyset \rightrightarrows \prod_\emptyset
\]
Thus every sheaf on $B$ vanishes.
\end{ex}

\begin{prop}\name{Sheaf of Functions}
\label{prop:gluemake}
If $X,Y$ are coarse spaces then the assignment $U\s X\mapsto$ (coarse maps $U\to Y$ modulo closeness) is a sheaf on $\coarsetopology X$.
\end{prop}
\begin{proof}
 We check the sheaf axioms:
 \begin{enumerate}
  \item global axiom: Let $f,g:U\to Y$ be two coarse maps and suppose $U$ is coarsely covered by $U_1,U_2$ and $f|_{U_1}\sim g|_{U_1}$ and $f|_{U_2}\sim g|_{U_2}$. Then 
\[
f\times g(\Delta_U)=f\times g(\Delta_{U_1})\cup f\times g(\Delta_{U_2})\cup f\times g(\Delta_{U\ohne(U_1\cup 
U_2)})
\]
The first two terms of the union are entourages because $f,g$ are close on $U_1$ and $U_2$. The last term is a entourage because $U\ohne (U_1\cup U_2)$ is bounded. Therefore $(f\times g)(\Delta_U)$ is a union of three entourages, so is itself an entourage. Thus $f,g$ are close on $U$.
\item gluing axiom: Suppose $U\s X$ is coarsely covered by $U_1,U_2$ and $f_1:U_1\to Y$ and $f_2:U_2\to Y$ are coarse maps such that
\[
 f_1|_{U_2}\sim f_2|_{U_1}.
\]
Then there is a global map $f:U\to Y$ defined in the following way:
\[
 f(x)=\begin{cases}
       f_1(x) & x\in U_1,\\
       f_2(x) & x\in U_2\ohne U_1,\\
       p & x\in U\ohne(U_1\cup U_2).
      \end{cases}
\]
Here $p$ denotes some point in $Y$. Now we show $f$ is a coarse map:

We show $f$ is coarsely uniform: If $E\s U^2$ is an entourage then
 \begin{enumerate}
 \item $\zz f(E\cap U_1^2)=\zz {f_1}(E\cap U_1^2)$ is an entourage;
 \item 
 \f{
 \zzp f{E\cap (U_1\cap U_2)\times (U_2\ohne U_1)}
 &=f_1\times f_2(E\cap (U_1\cap U_2)\times (U_2\ohne U_1))\\
 &\s f_1\times f_2(\Delta_{U_1\cap U_2})\circ \zzp {f_2}{E\cap (U_1\cap U_2)\times (U_2\ohne U_1)} 
 }
 is an entourage;
 \item $\zz f(E\cap (U_2\ohne U_1)^2)=\zz{f_2}(E\cap(U_2\ohne U_1)^2)$ is an entourage; 
 \item $E\cap U_1^c\times U_2^c$ and $E\cap U_2^c\times U_1^c$ are already bounded. Now $f$ maps bounded sets to bounded sets because $f_1,f_2$ and the constant map to $p$ do.
 \end{enumerate}
 Since
 \[
 U^2=U_1^2\cup (U_1\cap U_2)\times (U_2\ohne U_1)\cup (U_2\ohne U_1)\times (U_1\cap U_2)\cup (U_2\ohne U_1)^2\cup (U\ohne (U_1\cup U_2))^2
 \]
the set $\zzp f E$ is a finite union of entourages and therefore itself an entourage. Thus $f$ is coarsely uniform.
 
We show $f$ is coarsely proper: If $B\s Y$ is bounded then
\[
 \iip f B\s \iip {f_1} B\cup \iip {f_2} B \cup (U\ohne(U_1\cup U_2))
\]
is bounded.

Thus we showed $f$ is a coarse map.
\end{enumerate}
\end{proof}

\subsection{Sheaf Cohomology}

Sheaves on the Grothendieck topology $\coarsetopology X$ give rise to a cohomology theory on coarse spaces and coarse maps:

\begin{rem}
If $T$ is a Grothendieck topology denote by $\presheaves T$ the category of abelian presheaves on $T$ and by $\sheaves T$ the category of abelian sheaves on $T$. The category $\sheaves T$ is a full subcategory of $\presheaves T$, denote by $i:\sheaves T\to \presheaves T$ the inclusion functor. The functor $i$ is left exact by~\cite[Theorem~I.3.2.1]{Tamme1994}. If $U\in Cat(T)$ then denote by $\Gamma(U,\cdot):\presheaves T\to \abel$ the section functor which is an exact functor by~\cite[Proposition~I.2.1.1]{Tamme1994}. Then $\Gamma(U,\cdot)\circ i$ is additive and a composition of a left exact functor and an exact functor and therefore left exact. The category $\sheaves T$ is an abelian category with enough injectives therefore the right derived functor 
\[
\check H^q(U,\sheaff)=R^q(\Gamma(U,\cdot)\circ i)(\sheaff)
\]
exists for $\sheaff$ an abelian sheaf on $T$. See~\cite[Definition~I.3.3.1]{Tamme1994}.
\end{rem}

\begin{rem}\name{coarse cohomology with twisted coefficients}
 Let $\sheaff$ be a sheaf of abelian groups on a coarse space $X$, let $U\s X$ be a subset and let $q\ge 0$ be a number. Then the \emph{$q$th coarse cohomology group of $U$ with values in $\sheaff$} is
 \[
 \cohomology q U \sheaff,
\]
 the $q$th sheaf cohomology of $U$ in $\coarsetopology X$ with coefficient $\sheaff$.
\end{rem}

\begin{rem}\name{functoriality}
 Let $f: X\to Y$ be a coarse map between coarse space. There is a \emph{direct image functor} 
\f{
f_*: \sheaves {X_{ct}} &\to \sheaves {Y_{ct}}\\
\sheaff&\mapsto f_*\sheaff
}
where
\[
 f_*\sheaff(V)=\sheaff(\iip f V)
\]
for every $V\s Y$. The left adjoint functor to $f_*$ exists by \cite[Proposition~I.3.6.2]{Tamme1994} and is denoted {inverse image functor}
\[
 f^*:\sheaves Y \to \sheaves X.  
\]
 Note that $f^*$ is exact. Then there is an edge homomorphism of the Leray spectral sequence\footnote{This is~\cite[Theorem~I.3.7.6, p.~71]{Tamme1994}} of $f_*$ which will also be denoted by $f_*$: let $U\s Y$ be a subset and let $\sheaff$ be a sheaf on $X$; then there is a homomorphism
\[
 f_*:\cohomology * {\ii f U} \sheaff\to\cohomology * U {f_*\sheaff}.
\]
\end{rem}

\begin{rem}
Let $T$ be a Grothendieck topology. By~\cite[Theorem~I.3.1.1]{Tamme1994} the adjoint to the \emph{inclusion functor} $i:\sheaves T \to \presheaves T$ exists and is denoted by $\#$. If $\sheaff$ is a presheaf then $\sheaff^\#$ is the \emph{sheaf associated to the presheaf $\sheaff$}, also called the \emph{sheafification of $\sheaff$}.
 
 Define for an abelian presheaf $\sheaff$ on $T$:
 \[
 \sheaff^\nmid(U)=\lim_{\{U_i\to U\}_i\in Cov(T)}H^0(\{U_i\to U\},\sheaff)
 \]
 for $U\in Cat(T)$. Here the right side, the term $H^0(\{U_i\to U\},\sheaff)$, denotes the $0$th \v Chech cohomology associated to the covering $\{U_i\to U\}_i$ with values in $\sheaff$. The functor $\sheaff^\nmid$ is a presheaf and 
\[
\sheaff^{\#}=(\sheaff^\nmid)^\nmid
\]
is the sheaf associated to the presheaf $\sheaff$.
\end{rem}

\begin{lem} 
\label{lem:closepiso}
Let $X$ be a coarse space and denote by $p:X\times\{0,1\}\to X$ the projection to the first factor. Then
\[
 R^qp_*=0
 \]
 for $q>0$.
\end{lem}
\begin{proof}
 In a general setting if $\sheaff$ is a sheaf on a coarse space denote by $\sheafh^q(\sheaff)$ the presheaf
 \[
  U\mapsto \cohomology q U \sheaff.
 \]
 Then~\cite[Proposition~I.3.4.3]{Tamme1994} says that 
 \[
  \sheafh^q(\sheaff)^\dagger=0
 \]
 for $q>0$.

 Now~\cite[Proposition~I.3.7.1]{Tamme1994} implies that for every coarse map $f:X\to Y$ and sheaf $\sheaff$ 
on $X$
 \[
  R^qf_*(\sheaff)\cong (f_*\sheafh^q(\sheaff))^\#.
 \]

 Define
 \[
  H=\{((x,i),(x,0)):(x,i)\in X\times\{0,1\}\}\s (X\times \{0,1\})^2
 \]
 as a subset of $X\times\{0,1\}$ which is an entourage. We identify $X\times 0$ with $X$. Then $(U_i)_i$ coarsely covers $U\s X$ if and only if $(H[U_i])_i$ coarsely covers $H[U]$.
 
 Let $V_1, V_2$ be a coarse cover of $U\times\{0,1\}$. Write
 \[
  V_1=V_1^0\times 0\cup V_1^1\times 1
 \]
and 
  \[
  V_2=V_2^0\times 0\cup V_2^1\times 1.
 \]
 Note that
 \f{
 V_i^c
 &=(V_i^0\times 0)^c\cap (V_i^1\times 1)^c\\
 &=(V_i^0)^c\times 0\cup (V_i^1)^c\times 1
 }
 for $i=1,2$. But then
 \[
  ((V_1^0)^c\cup (V_1^1)^c)\times((V_2^0)^c\cup (V_2^1)^c)
 \]
 is a coentourage in $U$. Thus
 \[
  (V_1^0\cap V_1^1)\times\{0,1\}, (V_2^0\cap V_2^1)\times\{0,1\}
 \]
is a coarse cover that refines $V_1,V_2$.

 We show that $p_*$ and $\#$ commute for presheaves $\sheafg$ on $X$: Let $U\s X$ be a subset then
 \f{
 (p_*\sheafg)^\nmid(U)
 &=\lim_{\{U_i\to U\}_i\in Cov(X)}H^0(\{U_i\to U\}_i,p_*\sheafg)\\
 &=\lim_{\{U_i\to U\}_i\in Cov(X)}H^0(\{H[U_i]\to H[U]\}_i,\sheafg)\\
 &=\lim_{\{V_i\to H[U]\}_i\in Cov(X\times\{0,1\})}H^0(\{V_i\to H[U]\}_i,\sheafg)\\
 &=\sheafg^\nmid(H[U])\\
 &=p_*\sheafg^\nmid(U)
 }
\end{proof}

\begin{rem}
Note that two coarse maps $f,g:X\to Y$ are close if the map $h:X\times\{0,1\}\to Y$ agreeing with f on $X\times 0$ and with g on $X\times 1$ is a coarse map.
\end{rem}
\begin{proof}
Suppose $h$ is a coarse map we show $f,g$ are close. The set 
\f{
f\times g(\Delta_X)
&=\{f(x),g(x):x\in X\}\\
&=\{h((x,0),(x,1)):x\in X\}\\
&=\zz h(\Delta_X\times \{0,1\})
}
is an entourage in $Y$.

\end{proof}

\begin{thm}\name{close maps}
\label{thm:functor}
 If two coarse maps $f,g:X\to Y$ are close the induced homomorphisms $f_*,g_*$ of coarse cohomology with twisted coefficients are isomorphic.
\end{thm}
\begin{proof}
 Define a coarse map
 \[
  h:X\times \{0,1\}\to Y
 \]
by $h|_{X\times 0}=f$ and $h|_{X\times 1}=g$. But the inclusions $i_0:X\times 0\to X\times \{0,1\}$ and $i_1:X\times 1\to X\times \{0,1\}$ are both sections of the projection $p:X\times \{0,1\}\to X$ which by Lemma \ref{lem:closepiso} induces an isomorphism in coarse cohomology with twisted coefficients. Hence the maps $f=h\circ i_0$ and $g=h\circ i_1$ induce maps $f_*=h_*\circ i_{0*}$ and $g_*=h_*\circ i_{1*}$ which is the same map followed by isomorphisms.
\end{proof}

\begin{cor}\name{coarse equivalence}
Let $f:X\to Y$ be a coarse equivalence. Then $f$ induces an isomorphism in coarse cohomology with twisted coefficients.
\end{cor}

\subsection{Mayer-Vietoris Principle}
\label{subsec:MV}

In~\cite[Section~4.4, p.~24]{Shafarevich1996} a Mayer-Vietoris principle for sheaf cohomology on topological spaces is described. it can be translated directly to a Mayer-Vietoris principle for coarse spaces.

\begin{thm}\name{Mayer-Vietoris}
Let $X$ be a coarse space and $A,B$ two subsets that coarsely cover $X$. Then there is an exact sequence in cohomology
 \f{
 \cdots&\to \cohomology {i-1} {A\cap B}\sheaff\to \cohomology {i} {A\cup B}\sheaff\to\cohomology {i} 
{A}\sheaff\times\cohomology i B \sheaff\\
 &\to \cohomology {i} {A\cap B}\sheaff\to\cdots
 }
 for every sheaf $\sheaff$ on $X$.
\end{thm}
\begin{proof}
First note that a sheaf $\sheafg$ on a coarse space $X$ is called \emph{flabby} if the restriction map associated to an inclusion $U\to X$ is surjective for every $U\s X$. This implies that \v Cech cohomology $\cohomology q {\{U_i\to U\}_i} \sheafg =0$ for $q>0$ and every coarse cover $(U_i)_i$ of $U\s X$. Thus flabby sheaves are acyclic for coarse cohomology with twisted coefficients. Note also that every injective sheaf on a coarse space is flabby, thus given a sheaf $\sheaff$ there always exists a flabby resolution of $\sheaff$.

If $\sheafg$ is a flabby sheaf on $X$ the sequence
\[
 0\to \sheafg(A\cup B)\to \sheafg(A)\times\sheafg(B)\xrightarrow{\varphi}\sheafg(A\cap B)\to 0
\]
is an exact sequence. Here $\varphi$ sends a pair $(s_1,s_2)$ to $s_1|_{A\cap B}-s_2|_{A\cap B}$. Thus if $\sheaff$ is an arbitrary sheaf on $X$ there is an exact sequence of flabby resolutions of $\sheaff(A\cup B),\sheaff(A)\times\sheaff(B)$ and $\sheaff(A\cap B)$. This way we obtain the desired exact sequence in cohomology.
\end{proof}

\subsection{Local Cohomology}
\label{subsec:loccoh}

Let us define a version of relative cohomology for twisted coarse cohomology. There is already a similar notion for sheaf cohomology on topological spaces described in~\cite[chapter~1]{Hartshorne1967} which is called local cohomology. We do something similar:

\begin{defn}\name{support of a section}
 Let $s\in\sheaff(U)$ be a section. Then the support of $s$ is contained in $V\s U$ if 
 \[
  s|_{V^c\cap U}=0
 \]
\end{defn}

Let $X$ be a coarse space and $Z\s X$ a subspace. Then
\[
 \Gamma_Z(\sheaff):U\mapsto\ker(\sheaff(U)\to \sheaff(U\cap Z^c))
\]
is a sheaf on $X$.

\begin{thm}
 Let $Z\s X$ be a subspace of a coarse space and let $Y=X\ohne Z$. Then there is a long exact sequence
 \[
  0\to \cohomology 0 {U}{\Gamma_Z(\sheaff)}\to \cohomology 0 {U}\sheaff\to \cohomology 0 {U} 
{\sheaff|_Y}\to \cohomology 1 {U} {\Gamma_Z(\sheaff)}\to \cdots
 \]
for every subset $U\s X$ and every sheaf $\sheaff$ on $X$.
\end{thm}
\begin{proof}
 First we have an exact sequence
  \[
   0\to \Gamma_Z(\sheaff)\to \sheaff\to \sheaff|_Y
  \]
  and if $\sheaff$ is flabby we can write $0$ on the right.
 
 Let $\mathcal{I}=0\to\sheaff\to I_0\to I_1\to\cdots$ be an injective resolution of $\sheaff$. Note that every injective sheaf is flabby. Then there is an exact sequence of complexes
\[
 0\to \Gamma_Z(\mathcal I)\to \mathcal I\to \mathcal I|_Y\to 0
\]
which shows what we wanted to show.
 \end{proof}

\section{Constant Coefficients}
\label{sec:cc}

Before introducing a new coefficient for coarse cohomology with twisted 
coefficients we introduce a numerical invariant of coarse spaces which will be of interest when studying the coefficient.

\subsection{Number of Ends}

If a space is the coarse disjoint union of two subspaces we have a special case of a coarse cover. In \cite{Stallings1968} the number of ends of a group were studied; this notion can be generalized in an obvious way to coarse spaces.

This notation can also be found in \cite{Meier2008}:
\begin{defn}
 A coarse space $X$ is called \emph{one-ended} if for every coarse disjoint union $X=\bigsqcup_i U_i$ all but one of the $U_i$ are bounded.
\end{defn}

\begin{lem}
\label{lem:zplusoneended}
 The coarse space $\Z_+$, which consists of the non-negative integers, is one-ended.
\end{lem}
\begin{proof}
Suppose $\Z_+$ is the union of $U,V$ and $U,V$ are not bounded. Without loss of generality we can assume $U,V$ are a disjoint union. Now $(n)_{n\in \N}$ is a sequence where $(n)_{n\in \N}\cap U$ is not bounded and $(n)_{n\in \N}\cap V$ is not bounded.

For every $N\in \N$ there is a smallest $n\in U$ such that $n\ge N$ and there is a smallest $m\in V$ such that $m\ge N$. Without loss of generality $n$ is greater than $m$, then $(n,n-1)\in U\times V\cap E(\Z_+,1)$. Here $E(\Z_+,1)$ denotes the set of all pairs $(x,y)\in \Z_+^2$ with $d(x,y)\le 1$. This is an entourage. That way there is an infinite number of elements in
\[
 (U^2\cup V^2)^c\cap E(\Z_+,1)=(U\times V\cup V\times U)\cap E(\Z_+,1)
\]
which implies that $U,V$ are not coarsely disjoint.
\end{proof}

\begin{defn}
 Let $X$ be a coarse space. Its number of ends $\noends X$ is at least $n\ge 0$ if there is a coarse cover $(U_i)_i$ of $X$ such that $X$ is the coarse disjoint union of the $U_i$ and $n$ of the $U_i$ are not bounded.
\end{defn}

\begin{lem}
 If $A,B$ are two coarse spaces and $X=A\sqcup B$ their coarse disjoint union then
 \[
  \noends X=\noends A +\noends B.
 \]
\end{lem}
\begin{proof}
 Suppose $\noends A=n$ and $\noends B=m$. Then there are coarse disjoint unions $A=A_1\sqcup\ldots \sqcup A_n$ and $B=B_1\sqcup \ldots\sqcup B_m$ with nonboundeds. But then 
\[
X=A_1\sqcup\ldots \ldots \sqcup A_n\sqcup B_1\sqcup\ldots \sqcup B_m
\]
is a coarse disjoint union with nonboundeds. Thus $\noends X\ge \noends A+\noends B$.

Suppose $\noends X=n$. Then there is a coarse disjoint cover $(U_i)_{i=1,\ldots ,n}$ with nonboundeds of $X$. Thus $(U_i\cap A)_i$ is a coarse disjoint union of $A$ and $(U_i\cap B)_i$ is a coarse disjoint union of $B$. Then for every $i$ one of $U_i\cap A$ and $U_i\cap B$ is not bounded. Thus
\[
 \noends X \le \noends A +\noends B.
\]
\end{proof}

\begin{ex}
 $\noends \Z= 2$.
\end{ex}

\begin{thm}
\label{thm:esurjective}
 Let $f:X\to Y$ be a coarsely surjective coarse map and suppose $\noends Y$ is 
finite. Then
 \[
  \noends X \ge \noends Y.
 \]
\end{thm}
\begin{proof}

First we show that $\noends X \ge \noends{\im f}$: Regard $f$ as a surjective coarse map $X\to \im f$. Suppose that $\noends {\im f}=n$. Then $\im f$ is coarsely covered by a coarse disjoint union $(U_i)_{i=1,\ldots,n}$ where none of the $U_i$ are bounded. But then $(\iip f {U_i})_i$ is a coarse disjoint union of $X$ and because $f$ is a surjective coarse map none of the $\iip f {U_i}$ are bounded. 

Now we show that $\noends Y=\noends{\im f}$: Note that there is a surjective coarse equivalence $r:Y\to \im f$. By Proposition~\ref{prop:shiftcovers} a finite family of subsets $(U_i)_i$ is a coarse cover of $\im f$ if and only if $(\iip r {U_i})_i$ is a coarse cover of $Y$. if $(U_i)_i$ is a coarse disjoint union so is $(\iip r {U_i})_i$. 
\end{proof}

 \begin{cor}
 The number $\noends \cdot$ is a coarse invariant.
 \end{cor}
 
\subsection{Definition}
\label{subsec:definition}
\begin{defn}
 Let $X$ be a coarse space and $A$ an abelian group. Then $A_X$ (or just $A$ if 
the space $X$ is known from context) is the sheafification of the constant 
presheaf which associates to every subspace $U\s X$ the group $A$.
\end{defn}

\begin{lem}
 A coarse disjoint union $X=U\sqcup V$ of two coarse spaces $U,V$ is a coproduct in $\coarse$.
\end{lem}
\begin{proof}
 Denote by $i_1:U\to X$ and $i_2:V\to X$ the inclusions. We check the universal property: Let $Y$ be a coarse space and $f_1:U\to Y$ and $f_2:V\to Y$ coarse maps. But $U,V$ are a coarse cover of $X$ such that $U\cap V$ is bounded. Now we checked that already in Proposition~\ref{prop:gluemake}. The existence of a map $f:X\to Y$ with the desired properties would be the gluing axiom and the uniqueness modulo closeness would be the global axiom.
\end{proof}

\begin{thm}
\label{thm:constantcoeffs}
 Let $X$ be a coarse space and $A$ an abelian group. If $X$ has finitely many ends then
 \[
  A(X)=A^{e(X)}
 \]
and if $X$ does not have finitely many ends then
\[
 A(X)=\bigoplus_\N A.
\]
Here $A(X)$ means the evaluation of the constant sheaf $A$ on $X$ at $X$.
\end{thm}
\begin{proof}
By the equalizer diagram for sheaves a sheaf naturally converts finite 
coproducts into finite products. If $X$ is one-ended and $U,V$ a coarse cover 
of 
$X$ with nonboundeds then $U,V$ intersect nontrivially. Thus $A(X)=A$ in this 
case. If $X$ has infinitely many ends then there is a directed system
\[
\cdots \to U_1\sqcup\cdots\sqcup U_n\to  U_1\sqcup\cdots \sqcup U_{n+1}\to
\]
in the dual category of $\mathcal I_X$ which is the category of coarse covers of $X$. Here the $U_i$ are nonbounded and constitute a coarse disjoint union in $X$. Now we use \cite[Definition 2.2.5]{Tamme1994} by which
\[
 \cohomology 0 X A=\varinjlim_{(U_i)_i}H^0((U_i)_i,A).
\]
Then we take the direct limit of the system
\[
 \cdots\to A^n\to A^{n+1}\to A^{n+2}\to\cdots.
\]
Thus the result.

\end{proof}

\begin{lem}
\label{lem:zplusmanyended}
 If a subset $U\s \Z_+$ is one-ended then the inclusion
 \[
  i:U\to \Z_+
 \]
is coarsely surjective.
\end{lem}
\begin{proof}
 If the inclusion $i:U\to \Z_+$ is not coarsely surjective then there is an increasing sequence $(v_i)_i\s \Z_+$ such that for every $u\in U$:
 \[
  |u-v_i|>i.
 \]
Now define
\[
 A:=\{u\in U:v_{2i}<u<v_{2i+1},i\in\N\}
\]
and
\[
 B:=\{u\in U:v_{2i+1}<u<v_{2i},i\in\N\}.
\]
Then for every $a\in A,b\in B$ there is some $j$ such that $a<v_j<b$. Then
\f{
|a-b|&=|a-v_j|+|b-v_j|\\
&>2j.
}
If $i\in \N$ then $|a-b|\le i$ implies $a,b\le v_i$ Thus $A,B$ are a coarsely disjoint decomposition of $U$.
\end{proof}

\section{Remarks}

The starting point of this research was the idea to define sheaves on coarse spaces as presented in~\cite{Schmidt1999}. And then we noticed that cocontrolled subsets of $X^2$ which have first been studied in~\cite{Roe2003} have some topological features. 

Finally, after defining coarse covers which depend on the notion of 
coentourages, we came up with the methods of this paper. Note that coarse 
cohomology with twisted coefficients is basically just sheaf cohomology on the Grothendieck topology determined by coarse covers.

It would be possible, conversely, after a more thorough examination that other 
 cohomology and homology theories in the coarse category can be computed 
using sheaf cohomology tools. By \cite{Hartmann2017TCK} a modified version of controlled operator $K$-theory is a cosheaf on proper metric spaces. We obtain a new Mayer-Vietoris sequence using coarse covers. We also examined coarse cohomology by Roe and looked for sheaf properties. As of now we showed coarse cohomology in dimension $2$ is a sheaf on coarse spaces. It would be interesting to explore if coarse cohomology in dimension $2+q$ is a derived functor.

We wonder if the new sheaf cohomology will be of any help with understanding coarse spaces. We investigated in which way coarse covers determine a topology on a boundary of a coarse space. In~\cite{Hartmann2017b} we introduce a functor which assigns a uniform space with a coarse space. The uniformity is generated by coarse covers and coarse maps are mapped to uniformly continuous maps.

However, as of yet, we do not know many interesting sheaves on coarse spaces 
besides the constant sheaf and the sheaf of functions. It would be interesting 
to find another class of sheaves which can be defined for coarse spaces.

\bibliographystyle{halpha-abbrv}
\bibliography{mybib}

\begin{thebibliography}{{Har}19}
\expandafter\ifx\csname url\endcsname\relax
  \def\url#1{\texttt{#1}}\fi
\expandafter\ifx\csname doi\endcsname\relax
  \def\doi#1{\burlalt{doi:#1}{http://dx.doi.org/#1}}\fi
\expandafter\ifx\csname urlprefix\endcsname\relax\def\urlprefix{URL }\fi
\expandafter\ifx\csname href\endcsname\relax
  \def\href#1#2{#2}\fi
\expandafter\ifx\csname burlalt\endcsname\relax
  \def\burlalt#1#2{\href{#2}{#1}}\fi

\bibitem[Art62]{Artin1962}
M.~Artin.
\newblock {\em Grothendieck Topologies}.
\newblock Harvard University Press, Cambridge, Massachusetts, 1962.

\bibitem[BE17]{Bunke2017}
U.~{Bunke} and A.~{Engel}.
\newblock {Coarse cohomology theories}.
\newblock {\em ArXiv e-prints}, Nov. 2017,
  \burlalt{1711.08599}{http://arxiv.org/abs/1711.08599}.

\bibitem[DKU98]{Dranishnikov1998}
A.~N. Dranishnikov, J.~Keesling, and V.~V. Uspenskij.
\newblock On the {H}igson corona of uniformly contractible spaces.
\newblock {\em Topology}, 37(4):791--803, 1998.
\newblock \doi{10.1016/S0040-9383(97)00048-7}.

\bibitem[Gra06]{Grave2006}
B.~Grave.
\newblock {\em Coarse geometry and asymptotic dimension}.
\newblock PhD thesis, Georg-August Universität Göttingen, January 2006.

\bibitem[Har67]{Hartshorne1967}
R.~Hartshorne.
\newblock {\em Local cohomology}, volume 1961 of {\em A seminar given by A.
  Grothendieck, Harvard University, Fall}.
\newblock Springer-Verlag, Berlin-New York, 1967.

\bibitem[{Har}17]{Hartmann2017TCK}
E.~{Hartmann}.
\newblock {A twisted Version of controlled K-Theory}.
\newblock {\em ArXiv e-prints}, Nov. 2017,
  \burlalt{1711.03746}{http://arxiv.org/abs/1711.03746}.

\bibitem[{Har}19]{Hartmann2017b}
E.~{Hartmann}.
\newblock {A totally bounded uniformity on coarse metric spaces}.
\newblock {\em Topol. Appl.}, 2019.
\newblock \doi{10.1016/j.topol.2019.06.040}.

\bibitem[HR94]{Higson1994}
N.~Higson and J.~Roe.
\newblock A homotopy invariance theorem in coarse cohomology and {$K$}-theory.
\newblock {\em Trans. Amer. Math. Soc.}, 345(1):347--365, 1994.
\newblock \doi{10.2307/2154607}.

\bibitem[HR00]{Higson2000}
N.~Higson and J.~Roe.
\newblock {\em Analytic {$K$}-homology}.
\newblock Oxford Mathematical Monographs. Oxford University Press, Oxford,
  2000.
\newblock Oxford Science Publications.

\bibitem[HRY93]{Higson1993}
N.~Higson, J.~Roe, and G.~Yu.
\newblock A coarse {M}ayer-{V}ietoris principle.
\newblock {\em Math. Proc. Cambridge Philos. Soc.}, 114(1):85--97, 1993.
\newblock \doi{10.1017/S0305004100071425}.

\bibitem[Kee94]{Keesling1994}
J.~Keesling.
\newblock The one-dimensional \v{C}ech cohomology of the {H}igson
  compactification and its corona.
\newblock {\em Topology Proc.}, 19:129--148, 1994.

\bibitem[McL07]{Mclarty2007}
C.~McLarty.
\newblock The rising sea: {G}rothendieck on simplicity and generality.
\newblock In {\em Episodes in the history of modern algebra (1800--1950)},
  volume~32 of {\em Hist. Math.}, pages 301--325. Amer. Math. Soc., Providence,
  RI, 2007.

\bibitem[Mei08]{Meier2008}
J.~Meier.
\newblock {\em Groups, graphs and trees}, volume~73 of {\em London Mathematical
  Society Student Texts}.
\newblock Cambridge University Press, Cambridge, 2008.
\newblock \doi{10.1017/CBO9781139167505}.
\newblock An introduction to the geometry of infinite groups.

\bibitem[Mit01]{Mitchener2001}
P.~D. Mitchener.
\newblock Coarse homology theories.
\newblock {\em Algebr. Geom. Topol.}, 1:271--297, 2001.
\newblock \doi{10.2140/agt.2001.1.271}.

\bibitem[Roe03]{Roe2003}
J.~Roe.
\newblock {\em Lectures on coarse geometry}, volume~31 of {\em University
  Lecture Series}.
\newblock American Mathematical Society, Providence, RI, 2003.
\newblock \doi{10.1090/ulect/031}.

\bibitem[Sch99]{Schmidt1999}
A.~Schmidt.
\newblock Coarse geometry via {G}rothendieck topologies.
\newblock {\em Math. Nachr.}, 203:159--173, 1999.
\newblock \doi{10.1002/mana.1999.3212030111}.

\bibitem[Sha96]{Shafarevich1996}
I.~R. Shafarevich.
\newblock {\em Algebraic Geometry 2}.
\newblock Springer-Verlag, Berlin Heidelberg, 1996.

\bibitem[Sta68]{Stallings1968}
J.~R. Stallings.
\newblock On torsion-free groups with infinitely many ends.
\newblock {\em Ann. of Math. (2)}, 88:312--334, 1968.
\newblock \doi{10.2307/1970577}.

\bibitem[Tam94]{Tamme1994}
G.~Tamme.
\newblock {\em Introduction to \'etale cohomology}.
\newblock Universitext. Springer-Verlag, Berlin, 1994.
\newblock \doi{10.1007/978-3-642-78421-7}.
\newblock Translated from the German by Manfred Kolster.

\bibitem[Wei94]{Weibel1994}
C.~A. Weibel.
\newblock {\em An introduction to homological algebra}, volume~38 of {\em
  Cambridge Studies in Advanced Mathematics}.
\newblock Cambridge University Press, Cambridge, 1994.
\newblock \doi{10.1017/CBO9781139644136}.

\end{thebibliography}

\address

\end{document}